\documentclass{amsart}

\usepackage{tikz}
\usetikzlibrary{patterns}

\usepackage{enumerate}

\usepackage{geometry}
\newtheorem{thm}{Theorem}[section]
\newtheorem{lemma}[thm]{Lemma}
\newtheorem{prop}[thm]{Proposition}
\newtheorem{cor}[thm]{Corollary}

\theoremstyle{definition}

\theoremstyle{remark}

\newtheorem{remark}[thm]{Remark}

\numberwithin{equation}{section}

\DeclareMathOperator{\diver}{div}
\newcommand{\grad}{\nabla}
\newcommand{\hone}{\dot{H}^{1}}
\DeclareMathOperator{\im}{Im}
\newcommand{\lap}{\Delta}
\DeclareMathOperator{\loc}{loc}

\newcommand{\norm}[2][]{\left\| #2\right\|_{#1}}
\newcommand{\pd}[1][]{\partial_{#1}}
\newcommand{\reals}{\mathbb{R}}
\newcommand{\sphere}{\mathbb{S}}

\newcommand{\differential}[1]{\,d#1}
\newcommand{\dlam}{\differential{\lambda}}
\newcommand{\ds}{\differential{s}}
\newcommand{\dsigma}{\differential{\sigma}}
\newcommand{\dt}{\differential{t}}
\newcommand{\dw}{\differential{w}}
\newcommand{\dz}{\differential{z}}

\def \mr {{\mathbb R}}
\def \loc {{\operatorname{loc}}}

\def \wtb {{\widetilde{B}}}

\def \ms {{\mathbb S}}

\def \mcs {{\mathcal S}}

\def \mca {{\mathcal A}}
\def \mce {{\mathcal E}}

\def \mcl {{\mathcal L}}
\def \mcr {{\mathcal R}}

\def \mn {{\mathbb N}}

\def \la {\lambda}

\def \lan {\langle}   
\def \ran {\rangle}   
\def \del {\delta}

\def \whf {\widehat{F}}

\def \ha{ {\frac{1}{2}}}

\def \lap {\Delta}

\def \p {\partial}

\def \rao#1 {\frac{\p}{\p #1} #1}

 \newcommand{\supp}{\operatorname{supp}}

\numberwithin{equation}{section}

\usepackage{latexsym,eucal,amsfonts,amssymb,amsmath,graphicx}
\newtheorem{lem}{Lemma}

\newcommand{\beq}{\begin{equation}}
  \newcommand{\eeq}{\end{equation}}

\title{Radiation fields for semilinear wave equations}
\author[Dean Baskin]{Dean Baskin}
\address{Department of Mathematics, Northwestern University, Evanston, IL 60208-2730, U.S.A.}
\email{dbaskin@math.northwestern.edu}
\author[Ant\^onio S\'a Barreto]{Ant\^onio S\'a Barreto}
\address{Department of Mathematics, Purdue University, West Lafayette,
  IN 47907-1395, U.S.A.}
\email{sabarre@math.purdue.edu}
\date{December 19, 2012}

\begin{document}

\begin{abstract} We define the radiation fields of solutions to
  critical semilinear wave equations in $\mr^3$ and use them to define
  the scattering operator.  We also prove a support theorem for the radiation fields with radial initial data.  This  extends the well known support theorem for the Radon transform to this setting and can also be interpreted as a Paley-Wiener theorem for the distorted nonlinear Fourier transform of radial functions.
\end{abstract}

\maketitle

\section{Introduction}

In this paper we define the radiation field, describe its relationship
to the M{\o}ller wave operators, and prove a radial support theorem
for solutions of critical semilinear wave equations in $\reals^{3}$.
Our work is scattering-theoretic in nature; we rely on the previous
work of Grillakis~\cite{Grillakis:1990},
Shatah--Struwe~\cite{Shatah:1994}, and
Bahouri--G{\'e}rard~\cite{Bahouri:1999} to establish existence and
estimates for solutions of the equation.

We consider the following family of critical semilinear wave equations:
\begin{align}
  \label{CP}
  (\pd[t]^{2} - \lap )u + f(u) = 0 \;\ \text{ in } (0,\infty)\times \mr^3 \\
  u(0,z) = \phi (z), \; \pd[t]u(0,z) = \psi (z) \notag
\end{align}
We assume that $f(u)$ has the form $f(u) = u\cdot f_{0}(|u|^{2})$.

The equation~\eqref{CP} has a conserved energy:
\begin{equation*}
  E(t) = \frac{1}{2}\int_{\reals^{3}}\left( \left| \pd[t]u(t)\right|^{2}
    + \left| \grad u(t) \right|^{2}\right)\dz + \int_{\reals^{3}}P(u(t))\dz
\end{equation*}
Here $P(u)$ is notation for the potential energy term:
\begin{equation*}
  P(u) = \int_{0}^{u}f(s)\ds 
\end{equation*}

We further assume that the nonlinearity $f(u) = u\cdot f_{0}(|u|^{2})$ satisfies the following
hypotheses:
\begin{enumerate}
\item[(A1)] $f_{0}$ is smooth and real-valued,
\item[(A2)] $f_{0}(s) \geq 0$ for all $s \geq 0$,
\item[(A3)] $uf'(u)\sim f(u)$,
\item[(A4)] there are positive constants $c_{1}$ and $c_{2}$ so that
  \begin{equation*}
    c_{1}|u|^{5} \leq |f(u)| \leq c_{2} |u|^{5}, \text{ and}
\end{equation*}
\item[(A5)] the potential energy $P(u)$ is convex.
\end{enumerate}
We make these assumptions in order to ensure that the nonlinearity
$f(u)$ exhibits the same behavior as the power-type nonlinearity
$|u|^{4}u$.  The assumption that $|f(u)| \sim |u|^{5}$
implies that $P(u) \sim |u|^{6}$.  Note also that this class includes
many more functions than the power-type nonlinearities $c|u|^{4}u$.

Under such hypotheses, it is known from the work of Grillakis
\cite{Grillakis:1990} that if $\phi,\psi \in C_0^\infty(\mr^3),$ there
exists a unique solution to \eqref{CP} in $C^\infty(\mr_+ \times
\mr^3)$.\footnote{The proofs in the literature are typically specific to
  the case of the power-type nonlinearity but remain valid for
  nonlinearities of the above form.}  In this case, in section 3 of
\cite{Grillakis:1990}, Grillakis showed that the forward radiation
fields of $u$ exist.  In this paper we show that the radiation
fields exist for Shatah-Struwe solutions with finite energy initial data and show that they can be
used to obtain a formula for the scattering operator.  Moreover, we
prove a support theorem for the semilinear radiation fields.

For a solution of a wave equation, the radiation field is its rescaled
restriction to null-infinity.  In the linear Euclidean setting, if $u$ is a
solution of the wave equation on $\reals^{n+1}$, i.e., if $u$ solves
\begin{align*}
  \left( \pd[t]^{2} - \lap_{z}\right)u &= 0 \\
  (u,\pd[t]u)|_{t=0} &= (\phi, \psi) \in C^{\infty}_{0}(\reals^{n})
  \times C^{\infty}_{0}(\reals^{n}),
\end{align*}
then the radiation field of $(\phi, \psi)$ (written
$\mathcal{R}_{+}(\phi,\psi)$) is given by
\begin{equation*}
  \mathcal{R}_{+}(\phi,\psi)(s,\omega) = \lim _{r\to
    \infty}r^{\frac{n-1}{2}}(\pd[t]u)(s + r, r\omega).
\end{equation*}
Friedlander~\cite{Friedlander:1980,Friedlander:2001} showed that this
restriction is smooth and that it extends to an isometric isomorphism
from the energy space of initial data to $L^{2}$ on the cylinder:
\begin{equation*}
  \mathcal{R}_{+}: \dot{H}^{1}(\reals^{n}) \times L^{2}(\reals^{n})
  \to L^{2}(\reals \times \sphere^{n-1})
\end{equation*}
Moreover, the radiation field is a translation representation of the
wave group, i.e., it intertwines the wave group with translation on
the cylinder.  It is thus a concrete
realization of the translation representations central to Lax--Phillips
scattering theory \cite{Lax:1989} and is therefore connected with the
Radon transform.  For an overview of the radiation field and its
relationship to the Radon transform, we direct the reader to \cite{Lax:2001} or to the 
forthcoming manuscript of Melrose and Wang~\cite{Melrose:2011a}.

Radiation fields exist in a variety of geometric contexts, for example on asymptotically Euclidean manifolds,  asymptotically hyperbolic and  asymptotic complex hyperbolic spaces
\cite{Friedlander:1980,Friedlander:2001,Sa-Barreto:2003,Sa-Barreto:2005,Guillarmou-SB:2008}.    The Fourier transform of the radiation field is the adjoint of the  Poisson operator as defined in \cite{Melrose:1995}.  This can be also viewed as the distorted Fourier transform, see \cite{Hassell:1999}.  Support theorems for the radiation fields in these settings were proved in \cite{Sa-Barreto:2003,Sa-Barreto:2005,Guillarmou-SB:2008}.  
The
second author and Wunsch \cite{Sa-Barreto:2005a} further showed that for asymptotically Euclidean and asymptotically hyperbolic manifolds, the radiation field is a Fourier
integral operator associated to the graph of a sojourn relation.  In a
nonlinear setting, Wang~\cite{Wang:2011} studied the radiation field
for the Einstein equations on perturbations of Minkowski space with
spatial dimension $n\geq 4$.  

In the setting of the semilinear wave equation on $\reals^{3+1}$,  Grillakis \cite{Grillakis:1990} showed that the rescaled solution
may still be restricted to null infinity and so one may define the forward
\emph{nonlinear radiation field} for compactly supported smooth data
in the same way:
\begin{equation*}
  \mathcal{L}_{+}(\phi, \psi) (s,\omega) = \lim_{r\to\infty}r
  (\pd[t]u)(s+r, r\omega).
\end{equation*}
The backward nonlinear radiation field $\mathcal{L}_{-}(\phi, \psi)$
may be defined in the same way.  In this manuscript we show that
$\mathcal{L}_{\pm}$ are (nonlinear) isomorphisms of the space of
initial data with finite energy to $L^{2}(\reals \times\sphere^{2})$.
Moreover, we show that
\begin{equation*}
  \mathcal{L}_{\pm}: C^{\infty}_{0}(\reals^{3}) \times
  C^{\infty}_{0}(\reals^{3}) \to C^{\infty}(\reals \times \sphere^{2}).
\end{equation*}
The \emph{nonlinear scattering operator} $\mathcal{A}$ is given by
taking ``data at past null infinity'' to ``data at future null
infinity'' and is defined by
\begin{equation*}
  \mathcal{A} = \mathcal{L}_{+}\mathcal{L}_{-}^{-1}.
\end{equation*}

One would like to describe what type of operators $\mcl_+$ and $\mathcal{A}$ are and to understand how they propagate singularities.    Here we discuss what they do to functions on some Sobolev spaces. We have formulas for $\mcl_+$ and $\mca$  in Section \ref{sec:semilinear-case} in terms of the nonlinearity, but one would like to be able to say more in terms of the initial data. Among other results, we will show that for radial initial data $(0,\psi),$ $\psi\in C_0^\infty(\mr^3),$ the support of the radiation field controls the support of the initial data. We  prove the following:
\begin{thm}\label{supp01} Let $F(s)=\mcl_+(0,\psi),$ with $\psi\in C_0^\infty(\mr^3),$ radial.  If $F(s)=0$ for $s\leq -R,$ then $\psi(z)=0$ if $|z|\geq R.$
\end{thm}

We see in Section \ref{sec:radi-fields-line} that the linear
radiation field $\mcr_+(0,\psi)$ is given in terms of the Radon
transform of $\psi$ so this result can be interpreted as a support theorem for a generalized nonlinear Radon transform.  Moreover, in the linear equation the Fourier transform of the radiation field is given in terms
of the Fourier transform of the function $\psi.$  In the case of the linear Euclidean wave equation perturbed by a metric or by a potential, the Fourier transform of the forward radiation field $\widehat{\mcr_+}(0,\psi)$  is often called the distorted Fourier transform of $\psi.$ In this case, one can make sense of $\widehat {\mcl_+}(0,\psi)$ as the nonlinear distorted Fourier transform of $\psi$  which captures the effect of the
nonlinear potential.  Theorem \ref{supp01} thus can be viewed as a
Paley-Wiener type theorem for the nonlinear distorted Fourier
transform of $C_0^\infty$ radial functions.

In the linear case Theorem \ref{supp01} holds for $\psi \in L^2(\mr^3),$ radial. It is an open problem  whether this remains true for the semilinear equation. We can add some hypotheses and prove a result for $L^2$ initial data:
\begin{thm}
  \label{thm:fullsupportthm}
  If $F\in L^2(\reals),$ is compactly supported and satisfies $\int
  F(s)\ds = 0$, then, regarding $F$ trivially as a function of
  $\omega$,  $F = \mcl_{+}(\phi, \psi)$, where $\phi, \psi \in
  L^2(\mr^3)$ are compactly supported and radial.  If, moreover, $F
  \in C_0^\infty(\reals),$ then $\phi$ and $\psi$ are smooth, $\mca F$
  vanishes for $s$ sufficiently negative, and one can guarantee that
  $\phi(z)=\psi(z)=0$ for $|z|\geq R,$ where
  \begin{equation*}
    R=\min( \inf\supp F, \inf\supp \mca F).
  \end{equation*}
\end{thm}

\begin{remark}
  Notice that $F\in C_0^\infty(\mr)$ and $\int_\mr F(s) \; ds=0$ if
  and only if there exists $G\in C_0^\infty(\mr)$ such that $F(s)=
  G'(s).$ This is a dense subset of $L^2(\mr)$.
\end{remark}

Theorem~\ref{thm:fullsupportthm} is  weaker than the
corresponding statement in the linear setting---in the linear
setting, if $F(s)=0$ for $|s|\geq R,$ then   $\mcr_+(\phi,0)(s)=\mcr_+(0,\psi)(s)=0$ for $|s|\geq R,$ and the fact that $\int F(s) \ds=0$ would imply that 
$\int \mcr_+(\phi)(s) \ds=0.$  Together, these would show that both $\phi$ and $\psi$ are supported in $|z|\leq R.$

Another way of phrasing Theorem \ref{thm:fullsupportthm} is in terms of the
M{\o}ller wave operators.  It is now well known (see, e.g.,
Bahouri--G{\'e}rard~\cite{Bahouri:1999}) that the energy critical
semilinear wave equation exhibits scattering.  For a solution $u$ of
the nonlinear equation~\eqref{CP} that scatters to a solution $u_{+}$
of the linear equation, the wave operator $\Omega_{+}$ maps the
initial data for $u_{+}$ to the initial data for $u$.  $\Omega_{+}$ is
related to the radiation fields by
\begin{equation*}
  \Omega_{+} = \mathcal{L}_{+}^{-1}\mathcal{R}_{+}.
\end{equation*}
Theorem~\ref{thm:fullsupportthm} then states that if
$(\phi_{0},\psi_{0})$ are compactly supported, smooth, and radial,
then so are $\Omega_{+}(\phi_{0},\psi_{0})$.

Section~\ref{sec:radi-fields-line} of this paper defines the radiation
field for the linear inhomogeneous wave equation and describes its
properties.  Section~\ref{sec:semilinear-case} defines the radiation
field for the semilinear equation, while Section~\ref{sec:cont}
describes some mild continuity properties of the nonlinear radiation
field.  In Section~\ref{sec:asymp-comp} we describe its relationship
with the classical scattering and wave operators.
Section~\ref{sec:radf-smooth} contains an energy estimate showing that
the radiation field of compactly supported smooth data is itself
smooth, which is used in Section~\ref{sec:suppthm} to prove the
support theorem.  For completeness, we also include an appendix
containing a proof of the persistence of regularity for solutions of
the semilinear wave equation.

\subsection{Acknowledgements}
\label{sec:acknowledgements}

Both authors gratefully acknowledge NSF support. Baskin was supported
by postdoctoral fellowship DMS-1103436 and S\'a Barreto by grant DMS-0901334. We would
like to thank Rafe Mazzeo for fruitful discussions.

\section{Radiation fields for the non-homogeneous linear wave equation}
\label{sec:radi-fields-line}
As is standard, we define the homogeneous Sobolev space
\begin{equation*}
  \hone (\reals^{3}) = \{ \phi : \grad \phi \in L^{2}(\reals^{3})\}.
\end{equation*}
We also define the spaces
\begin{align*}
  H^{k}(\reals^{3}) &= \{ \phi : \pd[\alpha]\phi \in L^{2} \text{ for
    all } |\alpha| \leq k\} , \\
  \tilde{H}^{k}(\reals^{3}) &= \{ \phi : \pd[\alpha]\phi \in
  L^{2} \text{ for all } 1\leq |\alpha| \leq k \} ,
\end{align*}
with norms
\begin{align*}
  \norm[H^{k}]{\phi}^{2} &= \sum_{|\alpha|\leq k} \int_{\reals^{3}}
  \left| \pd[\alpha]\phi\right|^{2}\dz , \\
  \norm[\tilde{H}^{k}]{\phi}^{2} &= \sum_{1 \leq |\alpha| \leq k}
  \int_{\reals^{3}} \left| \pd[\alpha]\phi\right|^{2}\dz .
\end{align*}
The norms on $\tilde{H}^{k}$ differ from those on $H^{k}$ by the
absence of the $\norm[L^{2}]{\phi}$ component.

The energy norm of $(\phi , \psi) \in \hone(\reals^{3})\times L^{2}(\reals^{3})$ is defined to be
\begin{equation}
  E(\phi , \psi ) ^{2} = \int _{\reals^{3}}\left(|\grad _{z}\phi|^{2} + |\psi |^{2}\right) \dz. \label{linear-energy-norm}
\end{equation}
The higher energy norms $E_{k}(\phi, \psi)$ are defined to be
\begin{equation*}
  E_{k}(\phi, \psi)^{2} = \int _{\reals^{3}} \left( \sum_{1 \leq
      |\alpha| \leq k+1} \left| \pd[\alpha]\phi\right|^{2} +
    \sum_{|\alpha|\leq k} \left| \pd[\alpha] \psi\right|^{2}\right)
  \dz, 
\end{equation*}
i.e., $E_{k}(\phi, \psi)^{2} = \norm[\tilde{H}^{k+1}]{\phi}^{2} + \norm[H^{k}]{\psi}^{2}$.

We recall the definition of the radiation fields due to F.G. Friedlander, see \cite{Friedlander:2001} and references to his earlier work cited there, and how to obtain the scattering operator from them.

Given $f \in C^{\infty}_{0}((0,\infty)\times \reals^{3})$ and $\phi , \psi  \in C^{\infty}_{0}(\reals^{3})$, let $u \in C^{\infty}_{0}([0,\infty)\times\reals_{3})$ satisfy
\begin{align}
  \label{CP0}
  (\pd[t]^{2} - \lap )u = f \text{ in } (0,\infty) \times \reals^{3} \\
  u(0,z) = \phi (z), \;\ \pd[t] u(0,z) = \psi (z) . \notag
\end{align}

In what follows we will use the spaces $L^{r}(\reals ; L^{s}(\reals^{3}))$, $1\leq r, s\leq \infty$, with norm given by
\begin{equation*}
  \norm[L^{r};L^{s}]{F} = \norm[L^{r}(\reals)]{\norm[L^{s}(\reals^{3})]{F(t,\cdot)}}.
\end{equation*}

\begin{thm}
  \label{radiation0}
  Let $u$ satisfy   (\ref{CP0}),  with $\phi,\psi\in C_0^\infty(\mr^3)$ and $f\in C_0^\infty(\mr \times \mr^3).$
  Let $x = \frac{1}{|z|}$, $\theta = z / |z|$, and let $s_{+} = t - \frac{1}{x}$.  Then $v_{+}(x,s_{+},\theta) = x^{-1}u(s_{+}+\frac{1}{x}, \frac{1}{x}\theta) \in C^{\infty}([0,\infty)_{x} \times \reals_{s_{+}} \times \sphere^{2})$.  

  The forward radiation field, which is defined by
  \begin{equation*}
    \mathcal{R}_{+}(\phi,\psi,f)(s_{+},\theta) = \pd[s] v_{+}(0,s_{+},\theta),
  \end{equation*}
  exists and satisfies
  \begin{align}
    \label{radf1}
    \norm[L^{2}(\reals\times \sphere^{2})]{\mathcal{R}_{+}(\phi ,
      \psi, f)}&\leq E(\phi,\psi) + \norm[L^{1};L^{2}]{f}.
  \end{align}
\end{thm}

\begin{proof}
  Since $f\in C^{\infty}_{0}(\reals\times \reals^{3})$, the proof of
  \cite{Friedlander:2001} works for this case as well.  Since the
  energy of the solution $E(t) = E(u(t), \pd[t]u(t))$ satisfies $E(0)
  = E(\phi , \psi)$, by multiplying the equation (\ref{CP0}) by
  $\pd[t] u$ and integrating in $\reals^{3}$ we obtain
  \begin{align*}
    \frac{1}{2}\pd[t] (E(t))^{2} = \frac{1}{2}\int
    _{\reals^{3}}f(t,z)\overline{\pd[t]u(t,z)}\dz +
    \frac{1}{2}\int_{\reals^{3}}\pd[t]u(t,z)\overline{f}(t,z)\dz \leq
    \int _{\reals^{3}}\left| f(t,z)\pd[t]u(t,z)\right| \dz \leq \\
    \leq \norm[L^{2}(\reals^{3})]{f(t,\cdot)}
    \norm[L^{2}(\reals^{3})]{\pd[t]u(t,\cdot)} \leq
    \norm[L^{2}(\reals^{3})]{f(t,\cdot)}E(t).
  \end{align*}
  Hence $\pd[t]E(t) \leq \norm[L^{2}(\reals^{3})]{f(t,\cdot)}$ and thus
  \begin{equation*}
    E(t) \leq E(0) + \norm[{L^{1}([0,t];L^{2}(\reals^{3}))}]{f}.
  \end{equation*}
  In particular, this implies that for any $s_{0}$,
  \begin{equation*}
    \left[\int_{t-|z|\leq s_{0}}|\pd[t]u(t,z)|^{2}\dz\right]^{1/2} \leq E(0) + \norm[{L^{1}([0,t]; L^{2}(\reals^{3}))}]{f}.
  \end{equation*}
  Setting $v_{+}(x,s,\theta) = x^{-1}u(s+\frac{1}{x}, x, y)$ and
  taking the limit as $t\to \infty$ and then as $s_{0}\to \infty$, we obtain
  \begin{equation*}
    \norm[L^{2}(\reals\times \sphere^{2})]{\mathcal{R}_{+}(\phi , \psi, f)} \leq E(\phi, \psi) + \norm[L^{1}(\reals \times \reals^{3})]{f}.
  \end{equation*}
\end{proof}

By considering the solution in $t < 0$, and setting $v_{-} (x,s_{-} , \theta) = x^{-1}u(s_{-}- \frac{1}{x}, \frac{1}{x}\theta)$, then as above one can show that $v_{-} \in C^{\infty}([0,\infty)_{x}\times \reals_{s_{-}}\times \sphere^{2})$ and
\begin{equation*}
  \mathcal{R}_{-}(\phi, \psi, f) (s_{-}, \theta) = \pd[s]v_{-}(0,s_{-},\theta).
\end{equation*}
One can also show that \eqref{radf1} is satisfied for $\mathcal{R}_{-}$.

By linearity one can extend $\mathcal{R}_{\pm}$ as a continuous map
\begin{equation}
  \label{linradf}
  \mathcal{R}_{\pm} : \hone (\reals^{3}) \times L^{2}(\reals^{3}) \times L^{1}(\reals; L^{2}(\reals^{3})) \to L^{2}(\reals\times \sphere^{2}).
\end{equation}
We also know from \cite{Friedlander:2001} that the maps
\begin{align}
\begin{aligned}
  \mathcal{R}_{\pm}: \hone(\reals^{3}) \times L^{2}(\reals^{3}) \to L^{2}(\reals \times \sphere^{2}), \\
  (\phi, \psi) \mapsto \mathcal{R}_{\pm}(\phi, \psi, 0)
\end{aligned}\label{aux3}
\end{align}
are isometric isomorphisms.  It is worth mentioning that if $\phi$,
$\psi$, and $f$ are radial, the observation that $\mathcal{R}_{+}$
intertwines $\lap_{z}$ with $\pd[s]^{2}$ (which follows from properties of
the Radon transform) implies that
\begin{align*}
  \norm[H^{k}(\reals)]{\mathcal{R}_{+}(\phi, \psi,
    f)} &\leq E_{k}(\phi, \psi) + \norm[L^{1};H^{k}]{f},
\end{align*}  
and that
\begin{equation*}
  \mathcal{R}_{\pm}: \tilde{H}_{\text{rad}}^{k+1}(\reals^{3})\times
  H^{k}_{\text{rad}}(\reals^{3}) \to H^{k}\left( \reals\right)
\end{equation*}
are isometric isomorphisms.  

We will need the following:
\begin{prop}
  \label{control}
  Given $f\in L^{1}(\reals; L^{2}(\reals^{3}))$, there exist unique $(\phi , \psi) \in \hone(\reals^{3}) \times L^{2}(\reals^{3})$ such that
  \begin{equation*}
    \mathcal{R}_{+}(\phi , \psi , f) = 0,
  \end{equation*}
  and in this case
  \begin{equation*}
    E(\phi , \psi) \leq \norm[L^{1}(\reals ; L^{2}(\reals^{3}))]{f}.
  \end{equation*}

  Moreover, if $f\in L^{1}(\reals ; H^{k}(\reals^{3}))$ is radial,
  then there are unique $(\phi, \psi)\in
  \tilde{H}^{k+1}_{\text{rad}}(\reals^{3}) \times
  H^{k}_{\text{rad}}(\reals^{3})$ so that $\mathcal{R}_{+}(\phi, \psi, f)=0$.
\end{prop}

\begin{proof}
  To see this, one just needs to pick $(\phi, \psi)$ such that $\mathcal{R}_{+}(\phi, \psi, 0) = -\mathcal{R}_{+}(0,0,f)$.  But in this case, in view of Theorem~\ref{radiation0},
  \begin{equation*}
    E(\phi, \psi) = \norm[L^{2}]{\mathcal{R}_{+}(\phi, \psi, 0)} = \norm[L^{2}]{\mathcal{R}_{+}(0,0,f)}\leq \norm[L^{1}(\reals;L^{2}(\reals^{3}))]{f}.
  \end{equation*}

  To prove the second statement, we now assume that $f\in L^{1}H^{k}$
  is radial and so $F = -\mathcal{R}_{+}(0,0,f)$ is radial as well.
  Let $\phi$ and $\psi$ be such that $\mathcal{R}_{+}(\phi, \psi, 0) =
  F$.  Let $U$ be an orthogonal transformation, and let
  $U^*\phi=\phi\circ U,$ $U^*\psi=\psi\circ U.$ Since $U^*F=F\circ
  U=F,$ and the wave equation is invariant under orthogonal
  transformations, it follows that $F=\mcr_+(\phi,\psi,0)=\mcr_+(U^*
  \phi, U^* \psi,0).$ By uniqueness, $\phi=U^*\phi$ and $\psi=U^*
  \psi.$ We now use that for radial $\phi, \psi$, the fact that the
  radiation field intertwines $\lap_{z}$ with $\pd[s]^{2}$ implies
  that
  \begin{equation*}
    E_{k}(\phi, \psi) = \norm[H^{k}]{\mathcal{R}_{+}(\phi, \psi , 0)}
    = \norm[H^{k}]{\mathcal{R}_{+}(0,0,f)} \leq \norm[L^{1}(\reals ; H^{k}(\reals^{3}))]{f}.
  \end{equation*}
\end{proof}

It is useful to find a formula for $\mathcal{R}_{\pm}(\phi,\psi, f)$.
Since we are only interested in either the behavior of the solution
for positive or negative times, we multiply $u$ by the Heaviside
function in time.  If $u_+ = H(t)u$, and $F_{+}(t,z) = H(t)f(t,z)$, we
obtain
\begin{align*}
  (\pd[t]^{2} - \lap) u_+ &= \psi (z) \delta (t) + \phi(z) \delta ' (t) + F_+(t,z), \\
  u_{+} &= 0 \text{ for } t< 0.
\end{align*}
Taking the Fourier transform in $t$, we obtain
\begin{equation*}
  (\Delta+\lambda^{2} )\mathcal{F}(u_+) =  -\psi (z) - i \lambda \phi(z) - \mathcal{F}(F_{+})(\lambda, z),
\end{equation*}
where $\mathcal{F}$ denotes the Fourier transform in $t$.  Let $R_{-}(\lambda) = (\Delta+\la^2)^{-1}$ denote the resolvent of the Laplacian in $\reals^{3}$ which is holomorphic in $\im \lambda < 0$.  In this case it is well known that the kernel of $R_{-}(\lambda)$ is given by
\begin{equation}
  \label{resform}
  R_{-}(\lambda) (z,w) = \frac{1}{4\pi}\frac{e^{-i\lambda|z-w|}}{|z-w|}.
\end{equation}

We obtain
\begin{equation*}
  \mathcal{F}(u_+)(\lambda, z) = -R_{-}(\lambda) (\psi(z) + i\lambda \phi (z) + \mathcal{F}(F_{+})(\lambda, z)).
\end{equation*}
The Fourier transform in $s$ of the forward radiation field is given by
\begin{equation}
  \label{ftradf}
  \mathcal{F}(\mathcal{R}_{+}(\phi, \psi, f))(\lambda, \theta) = -\lim _{x\to 0} x^{-1}e^{i\lambda / x}R_{-}(\lambda)(i\lambda\psi (z) -  \lambda^{2} \phi(z) + i \lambda \mathcal{F}(F_{+})(\lambda,z)).
\end{equation}

But we deduce from (\ref{resform}) that if $x = \frac{1}{|z|}$, and $\theta = z / |z|$,
\begin{equation*}
  \lim _{x\to 0} x^{-1}e^{i\lambda/x}(R_{-}(\lambda)g)(\theta) = \frac{1}{4\pi} \int _{\reals^{3}}e^{i\lambda\langle \theta, w\rangle}g(w)\dw .
\end{equation*}
Applying this to (\ref{ftradf}), we obtain
\begin{equation*}
  \mathcal{F}(\mathcal{R}_{+}(\phi, \psi, f))(\lambda, \theta) = -\frac{i\la}{4\pi}\int _{\reals^{3}} e^{\la \langle \theta, z\rangle}\left( \psi (z) +i \lambda  \phi(z) +  \mathcal{F}(F_{+})(\lambda, z)\right)\dz.
\end{equation*}

Notice that
\begin{equation*}
  \int _{\reals^{3}}e^{i\lambda\langle \theta, w\rangle}\mathcal{F}(F_{+})(\lambda, w)\dw = \widehat{F_{+}}(\lambda , -\lambda \theta),
\end{equation*}
where $\widehat{F_{+}}$ is the Fourier transform of $F_{+}(t,w)$ in
both variables.  Therefore we obtain
\begin{equation}
  \label{formradf}
  \mathcal{F}(\mathcal{R}_{+}(\phi, \psi, f))(\lambda, \theta) = -\frac{i\la}{4\pi}\left(  \hat{\psi}(-\lambda\theta) +i \lambda \hat{\phi}(-\lambda\theta) + \widehat{F_{+}}(\lambda, -\lambda\theta)\right).
\end{equation}
As we remarked in the introduction, if $\phi$ and $f$ both vanish identically, then
\begin{equation*}
\mcr_+(0,\psi,0)=-\frac{i\la}{4\pi} \hat{\psi}(-\la\theta),
\end{equation*}
which is the essentially the Fourier transform of the function $\psi.$
In what follows, this construction will take into account the effect
of the nonlinear potential and can be viewed as a distorted form of
the Fourier transform of $\psi.$ In fact this is the name given to
this transformation for linear potential or metric perturbations of
the wave equation.  One should remark here that
$\whf_+(\la,-\la\theta)$ is well defined.
\begin{prop}\label{Fhat} If $F\in L^1(\mr; L^2(\mr^3)),$ then 
\begin{equation}
||\la\whf(\la,-\la\theta)||_{L^2(\mr \times \ms^2)} \leq ||F||_{L^1;L^2}.\label{boundfhat}
\end{equation}
\end{prop}
\begin{proof}  By density we only need to consider $F(t,z)=g(t) f(z),$ with $g\in L^1(\mr)$ and $f\in L^2(\mr^3).$ Then
\begin{equation}
\whf(\la,-\la\theta)= \widehat{g}(\la) \widehat{f}(-\la\theta).
\end{equation}
Therefore
\begin{equation}
||\la\whf(\la,-\la\theta)||_{L^2(\mr \times \ms^2)}^2= \int_{\mr} \int_{\ms^2} \la^2 |\widehat{g}(\la)|^2 |\widehat{f} (-\la\theta)|^2 \la^2 \; d\la d\theta.
\end{equation}
Since $g\in L^1(\mr),$ $\widehat{g}\in L^\infty(\mr)$ and $||\widehat{g}||_{L^\infty} \leq ||g||_{L^1}.$  Then, by Plancherel's Theorem we obtain
\begin{equation}
||\la\whf(\la,-\la\theta)||_{L^2(\mr \times \ms^2)}^2 \leq ||g||_{L^1(\mr)}^2 \int_{\mr} \int_{\ms^2}  |\widehat{f} (-\la\theta)|^2 \la^2 \; d\la d\theta \leq C ||g||_{L^1}^2|| f||_{L^2}^2.
\end{equation}
\end{proof}
To compute the backward radiation field, we replace $H(t)$ by $H(-t),$ set $u_-=H(-t) u$ and $F_-(t,z)=H(-t) F(t,z).$ Then
\begin{align*}
  (\pd[t]^{2} - \lap) u_- = -\psi (z) \delta (t) - \phi(z) \delta ' (t) + F_-(t,z), \\
  \tilde{u} = 0 \text{ for } t>0.
\end{align*}

define $F_{-}(t,z) = H(-t)f(t,z)$, and the replace the backward resolvent by the forward one,
\begin{equation*}
  R_{+}(\lambda) = \frac{1}{4\pi}\frac{e^{i\lambda|z-w|}}{|z-w|}, 
\end{equation*}
so we obtain
\begin{equation*}
  \mathcal{F}(\mathcal{R}_{-}(\phi, \psi, f))(\lambda, \theta) =
  \frac{i\la}{4\pi}\int _{\reals^{3}}
  e^{-i\lambda\langle\theta,w\rangle} \left( \psi (w) + i \lambda \phi
    (w) - \mathcal{F}(F_{-})(\lambda, w)\right) \dw,
\end{equation*}
and hence
\begin{equation}
  \label{formradf1}
  \mathcal{F}(\mathcal{R}_{-}(\phi, \psi, f))(\lambda, \theta) =
  \frac{i\la}{4\pi}\left(  \hat{\psi}(\lambda\theta) + i\lambda
    \hat{\phi}(\lambda\theta) - \widehat{F_{-}}(\lambda,
    \lambda\theta) \right). 
\end{equation}

We also want to express the radiation fields in terms of the Radon transform of the functions involved.  We recall that the Radon transform of a function $\phi (z)$ is defined by
\begin{equation*}
  R\phi(s,\theta) = \int _{\langle z, \theta\rangle = s} \phi (z) \dsigma (z), 
\end{equation*}
where $\dsigma (z)$ is the surface measure on the plane $\langle z, \theta\rangle = s$.  From this we deduce that the Fourier transform in the $s$ variable of $R(f)(s,\omega)$ is given by
\begin{equation*}
  \mathcal{F}(R\phi)(\lambda,\theta) = \int _{\reals^{3}}e^{-i\lambda\langle z,\theta\rangle}\phi(z)\dz = \hat{\phi}(\lambda\theta).
\end{equation*}
Notice that
\begin{equation}
  \int _{\reals}e^{i\lambda s}\widehat{F_{+}}(\lambda,-\lambda\omega)\dlam = \int _{\reals\times\reals\times \reals^{3}}e^{i\lambda(s-t+\langle\omega, z\omega\ran)}F_{+}(t,z) \dlam\dt\dz = \int _{t - \langle\omega,z\rangle = s}F_{+}(t,z)\dsigma (t,z). \label{invft}
\end{equation}
Now taking the inverse Fourier transform in $\lambda$ of (\ref{formradf})
and (\ref{formradf1}), we obtain
\begin{align}
  \begin{aligned}
    \mathcal{R}_{+}(\phi, \psi, f)(s,\theta) = -\frac{1}{4\pi}\p_s\left(
      R\psi (s,-\theta) + \pd[s] R\phi (s,-\theta) + \int _{t -
        \langle\theta, z\rangle = s} F_+(t,z)\dsigma (t,z)\right), \\
    \mathcal{R}_{-}(\phi, \psi, f)(s,\theta) =
    \frac{1}{4\pi}\pd[s]\left(R\psi(s,\theta) + \pd[s] R\phi
      (s,\theta)-\int _{t + \langle \theta, z\rangle = s}F_-(t,z)\dsigma
      (t,z)\right).
  \end{aligned}    \label{formula1}
\end{align}

\section{The radiation fields  for the semilinear wave equation}
\label{sec:semilinear-case}

Now we consider solutions to \eqref{CP} with finite energy Cauchy data. We recall the following result of Shatah and Struwe.

\begin{thm} \label{wellposed}[Shatah and Struwe \cite{Shatah:1993}]  
  For any $(\phi, \psi) \in \hone(\reals^{3}) \times L^{2}(\reals^{3})$, there exists a unique global (in time) solution to the problem~(\ref{CP}) such that $u$ is in the space
  \begin{equation*}
    X_{\loc} = C^{0}(\reals ; \hone (\reals^{3})) \cap C^{1}(\reals ; L^{2}(\reals^{3})) \cap L^{5}_{\loc}(\reals ; L^{10}(\reals^{3})).
  \end{equation*}
\end{thm}

Bahouri and G{\'e}rard \cite{Bahouri:1999} showed that the solution $u\in X_{\loc}$ to (\ref{CP}) in fact satisfies
\begin{equation}
  \label{bg1}
  u \in X = C^{0}(\reals ; \hone (\reals^{3})) \cap C^{1}(\reals ; L^{2}(\reals^{3})) \cap L^{5}(\reals ; L^{10}(\reals^{3})).
\end{equation}

The conserved energy is given by
\begin{equation}
\mce(\phi,\psi)= \frac{1}{2}\int_{\mr^3} \left( |\nabla_z
  \phi|^2  + |\psi|^2\right)\;\dz + \int_{\reals^{3}} P(\phi) \;\ dz. \label{nonlinear-energy-norm}
\end{equation}

Since $u\in X$, it follows that $|u|^{4}u \in L^{1}(\reals ;
L^{2}(\reals^{3}))$ and so $f(u)$ lies in the same space.  In view of
(\ref{linradf}), one can define the semilinear radiation fields as the
maps
\begin{align}
  \label{slradf}
  \mathcal{L}_{\pm} : \hone (\reals^{3}) \times L^{2}(\reals^{3}) \to L^{2}(\reals \times \sphere^{2}) \\
  \mathcal{L}_{\pm} (\phi, \psi) = \mathcal{R}_{\pm}(\phi , \psi, -f(u)). \notag
\end{align}
We deduce from (\ref{formula1}) that
\begin{gather}
  \begin{gathered}
    \mathcal{L}_{+}(\phi , \psi )(s,\theta) =
    -\frac{1}{4\pi}\pd[s]\left(R\psi (s,-\theta) + \pd[s]R\phi
      (s,-\theta) +\int_{t - \langle \theta, z\rangle = s}H(t) f(u(t,z))\dsigma (t,z)\right), \\
    \mathcal{L}_{-}(\phi , \psi )(s,\theta) =
    \frac{1}{4\pi}\pd[s]\left(R\psi (s,\theta) + \pd[s]R\phi
      (s,\theta)-\int_{t + \langle \theta, z\rangle = s}H(-t)
      f(u(t,z))\dsigma (t,z)\right),
  \end{gathered}\label{nonlinear-radf}
\end{gather}

Grillakis \cite{Grillakis:1990}  showed that the if the initial data is $C^\infty$ and has compact support, then the Friedlander radiation fields of the semilinear wave equation are in fact given by \eqref{nonlinear-radf}.  Following an idea of Bahouri and G{\'e}rard~\cite{Bahouri:1999} we prove

\begin{thm}\label{L2bound-radf}  The maps $\mcl_{\pm}(\phi,\psi)$ defined by \eqref{nonlinear-radf} are isomorphisms and
\begin{align}
\begin{aligned}
\mce(\phi,\psi)=|| \mcl_{\pm}(\phi,\psi) ||_{L^2(\mr\times \ms^2)}^2 
\end{aligned} \label{mcl-bound}
\end{align}
where $\mce$ is the nonlinear energy defined in
\eqref{nonlinear-energy-norm}. Moreover, they are translation
representations of the semilinear wave group.
\end{thm}
\begin{proof}
  The maps $\mcl_{\pm}$ are well defined, and the fact that they are
  translation representations of the semilinear wave group follows
  from the uniqueness of solutions to \eqref{CP} with finite energy.
  We want to show that their inverses are well defined as well, and
  that they satisfy \eqref{mcl-bound}.  We work with $\mcl_+.$
  The case of the backward radiation field is identical.
 
 We start by showing that  for any  $F\in L^{2}(\reals \times \sphere^{2})$ there exist a unique pair  $(\phi, \psi) \in \hone (\reals^{3}) \times L^{2}(\reals^{3})$ such that 
 $\mathcal{L}(\phi, \psi) = F,$ and moreover,
\begin{equation*}
\mce(\phi,\psi) = ||F||_{L^2(\mr \times \ms^2)}^2.
\end{equation*}

   We know there exists a unique pair $(\phi_{0},\psi_{0}) \in \hone(\reals^{3}) \times L^{2}(\reals^{3})$  such that $\mathcal{R}_{+}(\phi_{0}, \psi _{0}, 0) = F$ and in view of \eqref{aux3},
   \begin{gather}
   E(\phi_0,\psi_0)=   ||F||_{L^2(\mr \times \ms^2)}. \label{unit-linear}
   \end{gather}

   Let $v$ be the solution to (\ref{CP0}) with initial data $(\phi _{0}, \psi _{0})$ and $f=0.$  We then use the following Strichartz estimates due to Ginibre and Velo \cite{ginibre:1995}.
  \begin{thm}[Ginibre and Velo \cite{ginibre:1995}]
    \label{GV}
    Given $r\in [6,\infty)$, let $q$ satisfy
    \begin{equation*}
      \frac{1}{q} + \frac{3}{r} = \frac{1}{2}.
    \end{equation*}
    Then there exists $C_{r}$ such that for every $w(t,z)$ defined on $\reals\times \reals^{3}$,
    \begin{equation}
      \label{strichartzestimate}
      \norm[{L^{q}(\reals; L^{r}(\reals^{3}))}]{w} \leq C_{r}\left(\norm[L^{2}(\reals^{3})]{\grad_{t,z}w(0,z)} + \norm[{L^{1}(\reals; L^{2}(\reals^{3}))}]{(\pd[t]^{2}-\lap)w}\right).
    \end{equation}
  \end{thm}
  Therefore $v\in L^{5}(\reals; L^{10}(\reals^{3}))$, and thus for any $\delta > 0,$ one can choose $T_{0}$ such that
  \begin{equation}
    \norm[{L^{5}([T_{0},\infty); L^{10}(\reals^{3}))}]{v} < \delta. \label{choiceofdelta}
  \end{equation}
  Let $B_{\delta}$ denote the closed ball of radius $\delta$ in $L^{5}([T_{0}, \infty); L^{10}(\reals^{3}))$.  For $w \in B_{\delta}$ pick
  \begin{equation*}
    (\phi_{w}, \psi _{w}) \in \hone (\reals^{3})\times L^{2}(\reals^{3}) \text{ such that } \mathcal{R}_{+}(\phi _{w}, \psi _{w}, 0) (s+T_{0}, \theta) = -\mathcal{R}_{+}(0,0,-f(v+w))(s+T_{0}, \theta).
  \end{equation*}
  Notice that by the translation invariance, this corresponds to the
  solution of the Cauchy problem with data at $t= T_{0}$ instead of
  $t=0$.   In view of Theorem~\ref{radiation0} and the assumptions on
  $f$, we know that
  \begin{equation}
    E(\phi_{w}, \psi _{w}) \leq \norm[{L^{1}([T_{0},\infty);
      L^{2}(\reals^{3}))}]{f(v+w)} \leq C \norm[{L^{1}([T_{0}, \infty)
      ; L^{2}(\reals^{3}))}]{|v+w|^5} =  C\norm[{L^{5}([T_{0},\infty),
      L^{10}(\reals^{3}))}]{v+w}^5. \label{auxeq1} 
  \end{equation}
  Let $\tilde{w}$ be the solution of the Cauchy problem
  \begin{align*}
    (\pd[t]^{2} - \lap )\tilde{w} &= -f(v+w)\\
    \tilde{w}(T_{0}, z) = \phi_{w}(z), &\quad \pd[t]\tilde{w}(T_{0}, z) = \psi _{w}(z).
  \end{align*}

  But equation \eqref{strichartzestimate}  implies that
  \begin{equation*}
    \norm[{L^{5}([T_{0},\infty); L^{10}(\reals^{3}))}]{\tilde{w}} \leq
    C_{10}\left(E(\phi_{w}, \psi _{w}) + \norm[{L^{1}([T_{0}, \infty);
        L^{2}(\reals^{3}))}]{f(v+w)}\right) \leq
    2CC_{10}\norm[{L^{5}([T_{0},\infty);
      L^{10}(\reals^{3}))}]{v+w}^{5}. 
  \end{equation*}
  If we pick $\delta$ so small that\footnote{$\tilde{C}$ is defined
    below.}  $2C\cdot \tilde{C}\cdot
  C_{10}(3\delta)^{4} < \frac{1}{2}$, this defines a map
  \begin{align*}
    T: &B_{\delta} \to B_{\delta} \\
    &w \mapsto \tilde{w}.
  \end{align*}
  Moreover,  if $w_1,w_2 \in B_\del,$ equation~(\ref{strichartzestimate}) gives
  \begin{equation*}
    \norm[{L^{5}([T_{0},\infty); L^{10}(\reals^{3}))}]{Tw_{1} - Tw_{2}} \leq  
    \norm[{L^{1}([T_{0},\infty); L^{2}(\reals^{3}))}]{f(v+w_{1}) - f(v+w_{2})}.
  \end{equation*}
  For $\mu \in [0,1]$, let $h(\mu)= f\left( v+ \mu w_1+(1-\mu)
    w_2\right)$.  Then, there exists $\mu^*=\mu^*(t,z) \in [0,1]$ such
  that
  \begin{gather*}
    |h(1)-h(0)|= \left| f'\left( v + \mu^{*}w_{1} +
        (1-\mu^{*})w_{2}\right)\right| \cdot |w_{1}-w_{2}|.
  \end{gather*}
  Let us denote $\theta= f'(v + \mu^{*}w_{1} + (1-\mu^{*})w_{2})$. Therefore,
  using H\"older's inequality, with $p=5$ and $q=5/4$ we obtain
  \begin{align*}
    \norm[{L^{1}([T_{0},\infty);
      L^{2}(\reals^{3}))}]{f(v+w_{1}) - f(v+w_{2})}
    &\leq
    \int_\mr \left(\int_{\mr^3} | w_1-w_2|^2 |\theta|^2 \; dz\right)^\ha \; dt \leq \\
    \int_\mr \left(\int_{\mr^3} | w_1-w_2|^{10} \;
      dz\right)^{\frac{1}{10}} \left( \int_{\mr^3}
      |\theta|^{\frac{5}{2}} \; dz\right)^{\frac{2}{5}} \; dt .
  \end{align*}
  Using H\"older's inequality with the same exponents we obtain
  \begin{equation*}
    \norm[{L^{1}([T_{0},\infty); L^{2}(\reals^{3}))}]{f(v+w_{1}) - f(v+w_{2})} \leq 
    \norm[{L^{5}([T_{0},\infty); L^{10}(\reals^{3}))}]{w_1-w_2} 
    \left[ \int_\mr\left( \int_{\mr^3} |\theta|^{\frac{5}{2}} \; dz\right)^{\frac{1}{2}} \; dt\right]^{\frac{4}{5}}.
  \end{equation*}
  But there is a constant $\tilde{C}$ so that $|f'(u)| \leq
  \tilde{C}|u|^{4}$ for all $u$, so 
  \begin{equation*}
    \left[ \int_\mr\left( \int_{\mr^3} |\theta|^{\frac{5}{2}} \;
        dz\right)^{\frac{1}{2}} \; dt\right]^{\frac{4}{5}} \leq \tilde{C}
    \norm[{L^{5}([T_{0},\infty);
      L^{10}(\reals^{3}))}]{v+\mu^*w_1+(1-\mu^*) w_2}^{4} \leq
    \tilde{C}(3\del)^4. 
  \end{equation*}

  Therefore, with the choice of $\delta$ above, there exists a unique
  $w^{*}\in B_{\delta}$ such that $Tw^{*} = w^{*}$, and by
  construction
  \begin{equation}
    \mathcal{R}_{+}\left(\phi_{w^{*}} +v(T_{0}), \psi _{w^{*}} +
      \pd[t]v(T_{0}), -f(v+w^{*})\right)(s+T_{0},\theta) =
    F(s,\theta) \label{newaux} 
  \end{equation}
  and since $u = v+w^{*}$ satisfies
  \begin{align}
    \begin{aligned}
      (\pd[t]^{2} - \lap) u &= -f(u) \\
      u(T_{0}, z) = \phi_{w^{*}}(z) + v(T_{0}, z), &\quad \pd[t]u(T_{0},z) = \psi _{w^{*}}(z) + \pd[t]v(T_{0},z).
    \end{aligned}\label{auxCP1}
  \end{align}
  By the result of Shatah and Struwe \cite{Shatah:1994} this solution
  can be continued (uniquely in $X_{\loc}$) to all times $t<T_{0}$.  Therefore $(\phi, \psi) = (u(0,z), \pd[t]u(0,z))$ is the unique pair that satisfies $\mathcal{L}(\phi, \psi) = F.$ 
  
  The last step is to estimate the energy of the initial data in terms of $F.$    We know that $\mce(u(t),\p_t u(t))$ is conserved, so for any $t>0,$
  \begin{equation}
  \mce(\phi,\psi)= \mce(u(t),\p_t u(t))= \left( E(u(t),\p_t u(t))\right)^{2}+ \int_{\reals^{3}}P(u)(t,z) \; dz. \label{aux-cons}
  \end{equation}
  In particular, for $t=T_0,$
  \begin{equation}
  \mce(\phi,\psi)= \left( E(v(T_0)+\phi_{w^*}, \p_t v(T_0)+ \psi_{w^*})\right)^{2} + \int_{\mr^3} P(u)(t,z) \; dz. \label{aux-cons1}
  \end{equation}
  Since $E(v(T_0),\p_t v(T_0))= E(\phi_0,\psi_0)= ||F||_{L^2(\mr\times \ms^2)},$ we deduce from \eqref{auxeq1} that
  \begin{equation}
 \mce(\phi,\psi)=  ||F||_{L^2(\mr\times \ms^2)}^2 + \int_{\mr^3} P(u)(t,z)\; dz+ O(\del). \label{aux-cons2} 
  \end{equation}
  
  But the construction is independent of the choice of $T_0$ for which
  \eqref{choiceofdelta} is satisfied, and we know from the result of
  Bahouri and Shatah \cite{bahouri:1998} that
  \begin{gather*}
  \lim_{t\uparrow \infty} \int_{\mr^3} |u(t,z)|^6 \; dz=0 \quad \text{
    and therefore} \quad \lim_{t\uparrow\infty}\int_{\reals^{3}}P(u)(t,z)\;dz.
  \end{gather*}
  Since $\delta$ only depends on the constants $C_{10}$ (from the
  Strichartz estimate \eqref{strichartzestimate}), $C$, and
  $\tilde{C}$ (both from the nonlinearity), we obtain \eqref{mcl-bound}.
\end{proof}

In fact, for radial data, the maps $\mcl_{\pm}$ also preserve higher regularity.
\begin{thm}
  \label{thm:higer-reg-rad-field}
  If $(\phi, \psi) \in \tilde{H}^{k+1}(\reals^{3})\times
  H^{k}(\reals^{3})$, then $\mathcal{L}_{\pm}(\phi, \psi) \in
  H^{k}(\reals \times \sphere^{2})$.  Moreover, if $F\in H^{k}(\reals
  \times \sphere^{2})$ is radial, then $\mathcal{F} = \mcl_{\pm}(\phi,
  \psi)$ for $(\phi, \psi) \in \tilde{H}^{k+1}(\reals^{3})\times H^{k}(\reals^{3})$.
\end{thm}

\begin{proof}
  The proof of the first claim follows from persistence of regularity
  (Theorem~\ref{thm:higher-reg-NL}).  

  To prove the second claim, we repeat the iteration scheme in the
  proof of the previous theorem.  We start by noting that there is a
  unique $(\phi_{0}, \psi_{0}) \in \tilde{H}^{k+1}\times H^{k}$ so
  that $\mcr_{+}(\phi _{0}, \psi_{0}) = F$ and $E_{k}(\phi_{0},
  \psi_{0}) \leq C\norm[H^{k}]{F}$.  Let $v_{0}$ be the solution of
  equation~\eqref{CP0} with initial data $(\phi_{0}, \psi_{0})$ and
  vanishing inhomogeneous term.  We know by
  Proposition~\ref{prop:higher-order-energy-strichartz} that $v_{0}
  \in L^{5}W^{k,10}$ and so for any $\delta > 0$ there is a $T_{0}$ so
  that
  \begin{equation*}
    \norm[{L^{5}([T_{0}, \infty), W^{k,10})}]{v_{0}} < \delta.
  \end{equation*}
  We now repeat the scheme in the proof of the previous theorem,
  replacing all instances of $L^{5}L^{10}$ with $L^{5}W^{k,10}$.
  H{\"o}lder's inequality and the product rule allow us to
  estimate $\norm[L^{1}H^{k}]{(v+w)^{5}}$ in terms of
  $\norm[L^{5}W^{k,10}]{v}$ and $\norm[L^{5}W^{k,10}]{w}$.  We then
  obtain a contraction map (as before) of a small ball in
  $L^{5}W^{k,10}$.  Persistence of regularity (as in
  Theorem~\ref{thm:higher-reg-NL} then shows that $(\phi, \psi) \in
  \tilde{H}^{k+1}\times H^{k}$.  Uniqueness guarantees that these are
  the same $(\phi, \psi)$ as in Theorem~\ref{L2bound-radf}.
\end{proof}

\section{Some remarks about continuity}
\label{sec:cont}

In this section we show both the ``strong'' continuity of the
nonlinear radiation field and its ``norm'' continuity near zero.

\subsection{``Strong'' continuity}
We prove the following proposition:
\begin{prop}
  \label{prop:strong-cont}
  If $(\phi, \psi), (\phi _{n}, \psi_{n}) \in \dot{H}^{1}(\reals^{3})
  \times L^{2}(\reals^{3})$ and $(\phi_{n}, \psi_{n}) \to (\phi,
  \psi)$ in $\dot{H}^{1}\times L^{2}$, then $\mathcal{L}_{+}(\phi_{n},
  \psi_{n}) \to \mathcal{L}_{+}(\phi, \psi)$ in $L^{2}$.
\end{prop}

The proof of this proposition relies on a lemma due to Bahouri and
G{\'e}rard \cite{Bahouri:1999} that allows us to control uniformly the
decay of the $L^{6}$ norm of solutions.
\begin{lem}[Corollary 3 of \cite{Bahouri:1999}]
  \label{lem:bg99}
  Let $\mathcal{B}$ be a set of Shatah-Struwe solutions to
  equation~\eqref{CP}, with the following properties:
  \begin{equation*}
    \sup_{u\in\mathcal{B}} \left(\frac{1}{2} \int_{\reals^{3}} \left( \left|
        \grad u(0,z)\right|^{2} + \left| \pd[t]u(0,z)\right|^{2}
    \right) \dz + \frac{1}{6}\int_{\reals^{3}}\left|
      u(0,z)\right|^{6}\dz\right) < \infty,
  \end{equation*}
  and
  \begin{equation*}
    \lim_{R\to \infty}\sup _{u\in\mathcal{B}}\int_{|z|>R}\left|
      \grad_{t,z}u(0,z)\right|^{2} \dz = 0.
  \end{equation*}
  Then we have
  \begin{equation*}
    \lim_{|t|\to\infty}\sup_{u\in\mathcal{B}}\norm[L^{6}]{u(t,\cdot)}
    =0.
  \end{equation*}
\end{lem}

\begin{proof}[Proof of Proposition \ref{prop:strong-cont}]
  The main step in this proof is to show the ``strong'' continuity of
  the solution operator for equation~\eqref{CP} as an operator
  $\dot{H}^{1}\times L^{2}\to L^{5}(\reals; L^{10})$.  

  Let $u$ be the solution of equation~\eqref{CP} with initial data
  $(\phi, \psi)$ and $u_{n}$ the solution with data $(\phi_{n},
  \psi_{n})$.  

  Fix $\epsilon > 0$.  We claim that if $n$ is large enough then
  $\norm[L^{5}L^{10}]{u-u_{n}} \leq \epsilon$.

  Corollary 2 of \cite{Bahouri:1999} provides a constant $A$ so that
  \begin{equation*}
    \norm[L^{5}L^{10}]{u}, \norm[L^{5}L^{10}]{u_{n}} \leq A
  \end{equation*}
  for all $n$.  Note that the Strichartz estimates show that we may
  also assume (at the cost of replacing $A$ with a larger constant)
  that $A$ controls the $L^{4}L^{12}$ and $L^{\infty}L^{6}$ norms of
  $u$ and $u_{n}$.

  Because $(\phi_{n}, \psi_{n}) \to (\phi, \psi)$ in energy norm,
  Lemma~\ref{lem:bg99} shows that there is some $t_{0}$ so that for
  all $t > t_{0}$, and all $n$
  \begin{equation*}
    \norm[L^{6}]{u(t)} + \norm[L^{6}]{u_{n}(t)} \leq \frac{\epsilon}{A^{4}}.
  \end{equation*}

  Equation~\eqref{CP} is well-posed on arbitrarily long intervals, so
  there is some $N$ so that if $n > N$, we have
  \begin{equation*}
    E(u-u_{n})(t_{0}) + \norm[L^{5}({[-t_{0},t_{0}]};L^{10})]{u-u_{n}}
    < \epsilon.
  \end{equation*}

  Applying the Strichartz estimate again shows that
  \begin{equation*}
    \norm[L^{5}({[t_{0},\infty)}; L^{10})]{u-u_{n}} \leq C\left(
      E(u-u_{n})(t_{0}) +
      A^{4}\norm[L^{\infty}({[t_{0},\infty)};L^{6})]{u-u_{n}}\right)
    \leq 2C\epsilon,
  \end{equation*}
  where $C$ is independent of $\epsilon$ and $t_{0}$.  Combining this
  estimate with its counterpart on $(-\infty, -t_{0}]$ and the
  estimate above yields that if $n \geq N$,
  \begin{equation*}
    \norm[L^{5}L^{10}]{u-u_{n}} \leq 3C\epsilon.
  \end{equation*}
  This shows the ``strong'' continuity of the solution operator.
  
  We finally combine this estimate with the one in
  Theorem~\ref{radiation0} to find that
  \begin{align*}
    \norm[L^{2}]{\mathcal{L}_{+}(\phi, \psi) -
      \mathcal{L}_{+}(\phi_{n}, \psi_{n})} &=
    \norm[L^{2}]{\mathcal{R}_{+}(\phi - \phi_{n}, 0, 0) +
      \mathcal{R}_{+}(0, \psi - \psi_{n},0) -
      \mathcal{R}_{+}(0,0,f(u)-f(u_{n}))} \\
    &\leq \epsilon + (3C\epsilon)^{5},
  \end{align*}
  finishing the proof.
\end{proof}

\subsection{Continuity near zero}
We now show that the radiation field is continuous in a stronger sense
near $0$.  More precisely, we prove the following proposition.
\begin{prop}
  \label{prop:norm-cont}
  There is a $\gamma > 0$ such that the radiation field is a
  continuous map from $\{ (\phi, \psi)\in \dot{H}^{1}\times L^{2} : \mathcal{E}(\phi, \psi) <
  \gamma\}$ to $\{ F\in L^{2}(\reals \times \sphere^{2}):
  \norm[L^{2}]{F}<\gamma\}$.  

  In other words, for all $\epsilon > 0$ there is a $\delta > 0$ so
  that if $(\phi_{1},\psi_{1})$ and $(\phi_{2},\psi_{2})$ satisfy
  $\mathcal{E}(\phi_{i},\psi_{i}) < \gamma$ and
  $\norm[\dot{H}^{1}\times L^{2}]{(\phi_{1}-\phi_{2},
    \psi_{1}-\psi_{2})} < \delta$, then
  $\norm[L^{2}]{\mathcal{L}_{+}(\phi_{1},\psi_{1}) -
    \mathcal{L}_{+}(\phi_{2},\psi_{2}) } < \epsilon$.
\end{prop}

\begin{remark}
  \label{rem:inverse}
  Note that we may apply the inverse function theorem, together with
  this continuity and the unitarity of the linear radiation field at
  zero, to conclude that the inverse of $\mathcal{L}_{+}$ is also
  continuous at $0$.
\end{remark}

\begin{proof}
  We rely on a small-data variant of Corollary 2 of Bahouri and
  G{\'e}rard \cite{Bahouri:1999}.  In particular, we use that there is
  a constant $C$ and an $\gamma_{0} > 0$ so that if the energy of
  the initial data is bounded by $\gamma < \gamma_{0}$, then
  $\norm[L^{5}L^{10}]{u} \leq C\gamma$.  (In the language of
  Bahouri-G{\'e}rard, this is slightly stronger than the fact that
  $A(E) \to 0$ as $E\to 0$.)

  Using the above fact and the Strichartz estimates, we know that
  solutions have bounded $L^{4}L^{12}$ norm as well.  Let us call
  $A(\gamma)$ the constant that bounds the $L^{5}L^{10}$ and
  $L^{4}L^{12}$ norms of solutions with initial data having conserved
  energy bounded by $\gamma$.  The above note implies that $A(\gamma)
  \to 0$ as $\gamma \to 0$.
  
  We now fix $\epsilon >0$.  Suppose that $u$ and $v$ are two
  solutions having initial data energy with bounded by $\gamma$.
  Using the equation, we have that $u-v$ satisfies
  \begin{equation*}
    \Box (u-v) = f(v)-f(u).
  \end{equation*}
  Using the Strichartz estimate again (this time to bound the
  $L^{\infty}L^{6}$ norm of $u-v$) yields that
  \begin{equation*}
    \norm[L^{\infty}L^{6}]{u-v} \leq C \left( \norm[L^{2}]{\grad
        (u-v)(0)} + \norm[L^{2}]{\pd[t](u-v)(0)} +
      \norm[L^{\infty}L^{6}]{u-v}\left(
        \sum_{k=0}^{4}\norm[L^{4}L^{12}]{u}^{k}\norm[L^{4}L^{12}]{v}^{4-k}
      \right)\right).
  \end{equation*}
  In particular, the last term is bounded by $C \cdot A(\gamma)^{4}
  \cdot \norm[L^{\infty}L^{6}]{u-v}$ and so, if $\gamma$ is small, we
  have that
  \begin{equation*}
    \norm[L^{\infty}L^{6}]{u-v} \leq C' \left( \norm[L^{2}]{\grad
        (u-v)(0)} + \norm[L^{2}]{\pd[t](u-v)(0)}\right),
  \end{equation*}
  and so the solution operator is a continuous map from the ball of
  radius $\gamma$ in the energy space to a small ball in
  $L^{\infty}L^{6}$.  Applying the Strichartz estimates again shows
  that it is in fact continuous to a small ball in $L^{5}L^{10}$.

  Finally, we use equation~\eqref{radf1} to see that the $L^{2}$ norm of the difference of the
  radiation fields is bounded by the initial energy of $u-v$ and
  $\norm[L^{1}L^{2}]{f(u)-f(v)}$, both of which can be made
  arbitrarily small by making the initial energy of $u-v$ small.

  The final statement (that it maps a ball of radius $\gamma$ to a
  ball of radius $\gamma$) follows from the fact that the radiation
  field is norm-preserving  \eqref{mcl-bound}.
\end{proof}

\section{ Asymptotic Completeness and the scattering operator}  
\label{sec:asymp-comp}

  Let $(\phi,\psi) \in \hone (\reals^{3}) \times L^{2}(\reals^{3})
  $ and let $F=\mcl_+(\phi,\psi) \in L^{2}(\reals \times \sphere^{2}).$ The proof of Theorem \ref{L2bound-radf}  shows that
  if $(\phi_0,\psi_0)= \mcr_{+}^{-1} F,$ and $v$ satisfies
  \begin{gather*}
  \square v=0, \\
  v(0,z)=\phi_0, \;\ \p_t v(0,z)=\psi_0,
  \end{gather*}
 and $u$ is the solution to \eqref{CP} with initial data $(\phi,\psi),$ then for every $\del>0$ there exists $T_0$ such that
 \begin{gather*}
 E( u(T_0)-v(T_0), \p_t u(T_0)-\p_t v(T_0))< \del.
 \end{gather*}

Moreover,  the forward M{\o}ller wave operator 
\begin{gather*}
\Omega_+: \hone (\reals^{3}) \times L^{2}(\reals^{3}) \longrightarrow  \hone (\reals^{3}) \times L^{2}(\reals^{3}) \\
(\phi_0,\psi_0) \longmapsto  (\phi,\psi) 
\end{gather*}
is an isomorphism.  In fact

\begin{gather*}
\Omega_+(\phi,\psi)= \mcl_+^{-1} \mcr_+(\phi_0,\psi_0) \;\  \text{ and its inverse } \;\
\Omega_+^{-1}= \mcr_+^{-1} \mcl_+
\end{gather*}

Similarly, one can define the backward wave operator
\begin{gather*}
\Omega_-(\phi_0,\psi_0)= \mcl_-^{-1} \mcr_-(\phi_0,\psi_0) \;\  \text{ and its inverse } \;\
\Omega_-^{-1}= \mcr_-^{-1} \mcl_-.
\end{gather*}
  
The scattering operator is  defined to be the map
\begin{gather*}
\mcs:  \hone (\reals^{3}) \times L^{2}(\reals^{3}) \longrightarrow  \hone (\reals^{3}) \times L^{2}(\reals^{3}) \\
\mcs(\phi,\psi)= \Omega_+^{-1}  \Omega_-(\phi,\psi)
\end{gather*}
It follows from the discussion above that $\mcs$ preserves the energy norm $E.$  Notice that
\begin{equation}
\mcs= \mcr_+^{-1} \mcl_+ \mcl_-^{-1} \mcr_-. \label{modscatmat}
\end{equation}

We will follow Friedlander's definition of  the scattering operator and take it to be the map
\begin{align}
  \mathcal{A}:  \;&L^{2}(\reals\times \sphere^{2}) \to L^{2}(\reals \times \sphere^{2}),  \;\ \;\
  \mathcal{A} = \mathcal{L}_{+}\mathcal{L}_{-}^{-1}. \label{CP-E}
\end{align}

It follows from  \eqref{mcl-bound} that
\begin{equation}
||\mca F||_{L^2(\mr\times \ms^2)}=||F||_{L^2(\mr\times \ms^2)}. \label{unit-scat-mat}
\end{equation}

  Let $A_P$ denote the antipodal map on $\ms^2,$ i.e
\begin{gather*}
  A_P:\ms^2 \longrightarrow \ms^2 \\
  A_P(\theta)=-\theta.
\end{gather*}
Using the formulas for $\mcl_{\pm}$ from \eqref{nonlinear-radf} we obtain
\begin{gather*}
  A_P \mcl_-(\phi,\psi)(s,\theta)= \mcl_-(\phi,\psi)(s,-\theta)=
  -\mcl_+(\phi,\psi) -\frac{1}{4\pi} \p_s \int_{t-\lan z,\theta\ran=s}
  f(u)(t,z) \; d\sigma.
\end{gather*}
Since $\mcl_+=\mca \mcl_-,$ we obtain
\begin{gather}
  \mca \mcl_-= -A_P \mcl_- -\frac{1}{4\pi}\p_s \int_{t-\lan
    z,\theta\ran=s} f(u)(t,z) \; d\sigma. \label{scatmat1}
\end{gather}
In general,  given a function $F\in L^2(\mr \times \ms^2)$ there exists $(\phi, \psi) \in \hone(\reals^{3}) \times L^{2}(\reals^{3})$ such that
$\mcl_-(\phi,\psi)=F,$ and a unique $u \in  C^{0}(\reals ; \hone (\reals^{3})) \cap C^{1}(\reals ; L^{2}(\reals^{3})) \cap L^{5}(\reals ; L^{10}(\reals^{3}))$ satisfying \eqref{CP}.
Then it follows from \eqref{scatmat1} that
\begin{gather}
  \mca F= -A_P F -\frac{1}{4\pi} \p_s \int_{t-\lan z,\theta\ran=s} f(u)(t,z) \; d\sigma. \label{scatmat2}
\end{gather}

\section{The radiation fields for $C_0^\infty(\mr^3)$ data}
\label{sec:radf-smooth}

We will begin by proving the following strengthening of the results of
section 3 of \cite{Grillakis:1990}:

\begin{thm}\label{radiation1} Let $u$ be the solution to \eqref{CP} with $\phi,\psi\in C_0^\infty(\mr^3).$ Let $x=\frac{1}{|z|},$ $\theta=z/|z|,$ 
  $s_+=t-\frac{1}{x}$ and $s_-= t+\frac{1}{x}.$  If  
  $v_+(x,s_+,\theta)= x^{-1} u(s_++\frac{1}{x}, \frac{1}{x}\theta)$ and
  $v_-(x,s_-,\theta)= x^{-1} u(s_--\frac{1}{x}, \frac{1}{x}\theta),$ then
  $v_{\pm}\in C^\infty([0,\infty)_x \times \reals_{s_\pm} \times \sphere^2).$
  As above, the forward and backward semilinear radiation fields are defined to be
  \begin{equation*}
    \mathcal{L}_{+} ( \phi,\psi)(s_+,\theta)= \pd[s] v_+(0,s_+,\theta) \text{  and  }
    \mathcal{L}_-(\phi, \psi )(s_-,\theta)= \pd[s] v_-(0,s_-,\theta).
  \end{equation*}
\end{thm}

In what follows, we write $\tilde{f}(x,v) = x^{-5}f(xv)$.  Observe
that $\tilde{f}$ is a smooth function of $x$ and $v$ and that
$x^{-1}f(xv) = x^{4}\tilde{f}(x,v)$.  We also write $\tilde{P}(x,v) =
x^{-6}P(xv)$, where $P$ is antiderivative of $f$ defined above.  Note
that $\tilde{P}$ is also smooth, $\pd[v]\tilde{P} = \tilde{f}$, and
that $\pd[x]\tilde{P} = \frac{1}{x^{6}}\left( (xv)P'(xv) -
  6P(xv)\right)$, which is nonnegative by our assumptions on $f$ and
$P$.

\begin{proof}
  Work with the forward radiation field.  This of course proves the
  result for backward one as well.  We recall two facts from
  \cite{Grillakis:1990}.  First $u\in C^\infty(\mr_+ \times \mr^3)$
  and secondly, equation (3.1) of \cite{Grillakis:1990} gives that
  \begin{equation}
    (t^2-|z|^2) |u(t,x)| \leq C(\phi,\psi). 
    \label{upper1}
  \end{equation}
  
  Since $\phi,\psi\in C_0^\infty(\mr),$ let us assume that 
  \begin{gather*}
    \phi(z)=\psi(z)=0 \text{ if } |z|\geq R.
  \end{gather*}
  Finite speed of propagation then implies that
  \begin{align}
    \begin{aligned}
      u(t,z)=0 \text{ if }  t-|z|\leq -R, \;\ \text{ and } t>0 \text{ and  } \\
      u(t,z)=0 \text{ if }  t+|z|\geq R, \;\ \text{ and } t<0.
    \end{aligned}\label{finitespeed1}
  \end{align}
  In terms of coordinates $s_+$ and $s_-$ this implies that
  \begin{gather}
    \begin{gathered}
      u(t,z)=0 \text{ if }  s_+ \leq -R, \;\ \text{ and } s_+ + s_->0 \text{ and  } \\
      u(t,z)=0 \text{ if }  s_- \geq R, \;\ \text{ and }  s_++s_- < 0.
    \end{gathered}\label{finitespeed}
  \end{gather}

  Since we are working in three dimensions, the Euclidean Laplacian written in polar coordinates 
  $(r,\theta),$ $r=|z|$ and $\theta=z/|z|$ is given by
  \begin{equation*}
    \Delta= \pd[r]^2 +\frac{2}{r} \pd[r]+ \frac{1}{r^2} \Delta_{\sphere^2}.
  \end{equation*}
  Setting $x=\frac{1}{r},$  and $u= x v$   we get that $v$ satisfies
  \begin{equation*}
    (\pd[t]^2 -(x^2\pd[x])^2 - x^2 \Delta_{\sphere^2})v + x^{-1}f(xv)=0.
  \end{equation*}
  We compactify $s_+$ and $s_-$ by setting
  \begin{equation*}
    \mu=-\frac{1}{s_+}  \text{ and } \nu=\frac{1}{s_-}.
  \end{equation*}
  In terms of $t$ and $x$ we have
  \begin{gather*}
    2t= \frac{1}{\nu}-\frac{1}{\mu}= \frac{\mu-\nu}{\mu\nu} \text{ and }   \frac{2}{x}= \frac{1}{\nu}+\frac{1}{\mu}= \frac{\mu+\nu}{\mu\nu}.
  \end{gather*}

  The Cauchy problem \eqref{CP} with initial data $\phi$ and $\psi$
  translates into (recall that $\theta$ is a variable on the unit
  sphere and $\tilde{f}(x,v) = x^{-5}f(xv)$)
  \begin{gather}
    \begin{gathered}
      \left( (\mu+\nu)^2\pd[\mu] \pd[\nu] +\Delta_{\sphere^2}\right)
      v- \left( 2\frac{\mu\nu}{\mu+\nu}\right)^{2} \tilde{f}\left(
        \frac{2\mu\nu}{\mu+\nu},v\right)=0, \text{
        in }  (0,T)\times (0,T)\times \sphere^2 \\
      v(\mu,\mu,\theta)= \phi \left(\frac{1}{\mu}\theta \right), \;\
      (\pd[\mu] v)(\mu,\mu,\theta)= \frac{1}{2} \left(
        \frac{1}{\mu^{3}} \psi \left( \frac{1}{\mu}\theta
        \right)+\frac{1}{\mu} \pd[\mu] \phi\left( \frac{1}{\mu}\theta
        \right) - \frac{1}{\mu^{2}}\phi \left(\frac{1}{\mu}\theta
        \right) \right).
    \end{gathered} 
    \label{CP2}
  \end{gather}
  Here we used that $\pd[t] u = x \pd[t] v= \frac{2\mu\nu}{\mu+\nu}
  (\mu^2\pd[\mu]-\nu^2\pd[\nu]) v$ and that $\pd[\mu]
  v(\mu,\mu,\theta)= (\pd[\mu] v)(\mu,\mu,\theta)+ (\pd[\nu]
  v)(\mu,\mu,\theta).$ This also implies that
  \begin{equation}
    \pd[\nu] v(\mu,\mu,\theta)= \frac{1}{2} \left( \frac{1}{\mu}
      \pd[\mu]\left( \phi \left( \frac{1}{\mu} \theta \right)\right) -
      \frac{1}{\mu^{2}} \phi \left( \frac{1}{\mu}\theta \right) -
      \frac{1}{\mu^{3}} \psi \left( \frac{1}{\mu}\theta \right)\right).
    \label{cond1}
  \end{equation}
  
  Equations \eqref{upper1} and \eqref{finitespeed} translate into
  \begin{align}
    \begin{aligned}
      |v(\mu,\nu,\theta)| \leq C (\mu+\nu) \text{ and } \\
      v=0 \text{ if } \mu\leq \frac{1}{R} \text{ and } \nu\leq \frac{1}{R}. 
    \end{aligned}\label{finitespeed2}
  \end{align}

  First we obtain the following energy estimates:
  \begin{lemma}
    \label{enerest} 
    Let $\mu_0>0$ and let $w \in C^\infty((0,T)\times (0,T) \times
    \sphere^2),$ supported in $\{\mu\geq \mu_0\}\cup \{\nu\geq
    \mu_0\}$, and let $G \in L^\infty((0,T)\times (0,T)
    \times \sphere^2)$ and $F \in L^2((0,T)\times (0,T) \times
    \sphere^2)$ be such that
    \begin{gather}
      \begin{gathered}
        \left( (\mu+\nu)^2\pd[\mu] \pd[\nu] +\Delta_{\sphere^2}\right) w - G(\mu,\nu,\theta) w=F(\mu,\nu,\theta), \text{ and we denote } \\
        w(\mu,\mu,\theta)= w_0(\mu,\theta), \;\ (\pd[\mu] w)(\mu,\mu,\theta)= w_1(\mu,\theta).
      \end{gathered}  \label{diffweq}
    \end{gather}
    For $0<a<b<T,$ let 
    \begin{gather}
      \begin{gathered}
        \Omega_{ab}=\{ (\mu,\nu):  \nu\geq \mu_0, \;\ a \leq \mu \leq \nu \leq b\}, \text{ and denote its boundaries by}\\
        \Sigma_1= \{ (\mu,\nu): \mu=a,  \max{(\mu_0, a)} \leq \nu\leq b\}, \;\
        \Sigma_2= \{ (\mu,\nu): \nu=b, a \leq \mu\leq b\}, \\
        \Sigma_3= \{ (\mu,\nu): \nu=\mu, \max{(a,\mu_0)} \leq \mu\leq b\} \\
        \Sigma_{\mu_0}= \emptyset \text{ if }  \mu_0<a \text{ and } \Sigma_{\mu_0}= \{ \nu=\mu_0, \;\ a \leq \mu \leq \mu_0\}, \text{ if } \mu_0\geq a
      \end{gathered}      \label{regionab}
    \end{gather}
    
    Then there exists a constant $C$  which depends on $\mu_0$, $T$, and on $||G||_{L^\infty}$ such that
    \begin{align}
      \begin{aligned}
        \norm[L^{2}(\Omega_{ab})]{\pd[\nu]w}^{2} &+
        \norm[L^{2}(\Omega_{ab})]{\pd[\mu]w}^{2} +
        \norm[L^{2}(\Omega_{ab})]{\grad_{\sphere^{2}}w}^{2} \\
        \leq & C \left( E(w_{0},w_{1}) + R(w, \mu_{0}) +
          \norm[L^{2}(\Sigma_{3}\times \sphere^{2})]{w}^{2} +
          \norm[L^{2}(\Omega_{ab}\times \sphere^{2})]{F}^{2}\right),
      \end{aligned}   \label{ineq}
    \end{align}
    where 
    \begin{align*}
      E(w_0,w_1)= \frac{1}{2} \int_{\Sigma_3 \times \sphere^2} \left(
        |w_1|^2 +
        |\pd[\nu] w_0|^2+ \frac{1}{2} \mu^{-2}|\nabla_{\sphere^2} w_0|^2\right) \; d\mu d\theta \text{ and } \\
      R(w,\mu_0)= 0 \text{ if } a\geq \mu_0, R(w,\mu_0)=
      \norm[L^2(\Sigma_{\mu_0} \times \sphere^2)]{\pd[\mu]w}^{2}
      +\norm[L^{2}(\Sigma_{\mu_{0}}\times \sphere^{2})]{(\mu +
        \mu_{0})^{-1}\grad_{\sphere^{2}}w}^{2}, \text{ if } a<\mu_0.
    \end{align*}

    Moreover, if we let $a = 0, b = T$, then we may remove
    $\norm[L^{2}]{w}^{2}$ from the right side:
    \begin{align}
      \begin{aligned}
        \norm[L^{2}(\Omega_{0T})]{\pd[\nu]w}^{2} &+
        \norm[L^{2}(\Omega_{0T})]{\pd[\mu]w}^{2} +
        \norm[L^{2}(\Omega_{0T})]{\grad_{\sphere^{2}}w}^{2} \\
        \leq & C \left( E(w_{0},w_{1}) + R(w, \mu_{0}) +
          \norm[L^{2}(\Omega_{0T}\times \sphere^{2})]{F}^{2}\right).
      \end{aligned}  
    \end{align}
  \end{lemma}

  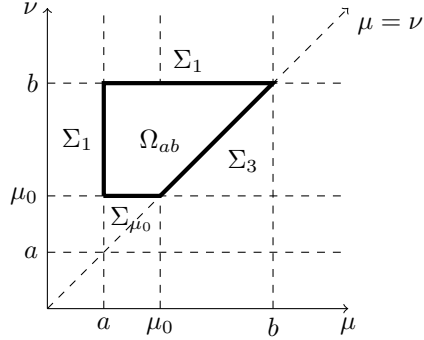
\begin{figure}[htp]
    \centering
    \begin{tikzpicture}
      \draw[->] (0,0) -- (4,0) node [anchor=north]{$\mu$};
      \draw[->] (0,0) -- (0,4) node [anchor=east]{$\nu$};
      \draw[->, dashed] (0,0) -- (4,4) node [anchor=north west]{$\mu = \nu$};

      \draw[dashed] (1.5,0) node [anchor=north]{$\mu_{0}$}-- (1.5,4);
      \draw[dashed] (0,1.5) node [anchor=east]{$\mu_{0}$} -- (4,1.5);

      \draw[dashed] (0.75, 0) node[anchor=north]{$a$} -- (0.75, 4);
      \draw[dashed] (0, 0.75) node[anchor=east]{$a$} -- (4, 0.75);
      \draw[dashed] (3,0) node[anchor=north]{$b$} -- (3,4);
      \draw[dashed] (0,3) node[anchor=east]{$b$} -- (4,3);

      \draw [ultra thick, fill=white] (0.75, 1.5) -- (1.5,1.5) node[pos=0.5,
      anchor=north]{$\Sigma_{\mu_{0}}$} -- (3,3) node[pos=0.5, anchor =
      north west]{$\Sigma_{3}$} -- (0.75, 3) node[pos=0.5,
      anchor=south]{$\Sigma_{1}$} -- (0.75, 1.5) node[pos=0.5,
      anchor=east]{$\Sigma_{1}$};

      \node[anchor=north] at (1.5,2.5) (label){$\Omega_{ab}$};

    \end{tikzpicture}
    \caption{The region $\Omega_{ab}$.}
    \label{fig:omegaab}
  \end{figure}

  \begin{proof}
    We multiply equation~\eqref{diffweq} by
    $\frac{1}{2}(\mu+\nu)^{-2}(\pd[\mu]-\pd[\nu]) \overline{w}$ and
    add the result to its complex conjugate to obtain the following
    identity
    \begin{align}
      \begin{aligned}
        \frac{1}{2} \pd[\nu]\left[ |\pd[\mu] w|^2 +
          (\mu+\nu)^{-2}|\nabla_{\sphere^2} w|^2 \right]- \frac{1}{2}
        \pd[\mu]\left[|\pd[\nu] w|^2 +
          (\mu+\nu)^{-2}|\nabla_{\sphere^2} w|^2  \right]+ \\
        \frac{1}{2}\diver_{\sphere^2}\left[ \left(
            (\mu+\nu)^{-2}(\pd[\mu]-\pd[\nu]) \overline{w}\right)
          \nabla_{\sphere^2} w + \left((\mu+\nu)^{-2}(\pd[\mu] -
            \pd[\nu])w\right)\grad_{\sphere^{2}}\overline{w}\right] \\
        = \frac{1}{2}(\mu + \nu)^{-2}\left( G
          w(\pd[\mu]-\pd[\nu])\overline{w} + \overline{Gw}(\pd[\mu] -
          \pd[\nu])w\right)+ \frac{1}{2}(\mu + \nu)^{-2}\left(
          F(\pd[\mu]-\pd[\nu])\overline{w} + \overline{F}(\pd[\mu] -
          \pd[\nu])w\right).
      \end{aligned}      \label{multid}
    \end{align}

    Then we integrate \eqref{multid} in $\Omega_{ab}\times \sphere^2$ and we find
    \begin{align}
      \begin{aligned}
        \frac{1}{2} \int_{\Sigma_1\times \sphere^2} \left[ |\pd[\nu]
          w|^2 + (\mu+\nu)^{-2}|\nabla_{\sphere^2} w|^2 \right]\; d\nu
        d\theta+ \frac{1}{2} \int_{\Sigma_2\times \sphere^2} \left[
          |\pd[\mu] v|^2 + (\mu+\nu)^{-2}|\nabla_{\sphere^2} v|^2
        \right]\; d\mu
        d\theta  \\
        = -\Re\int_{\Omega_{ab}\times \sphere^2 } G w
        (\mu+\nu)^{-2}(\pd[\mu]-\pd[\nu])\overline{w} \; d\mu d\nu
        d\theta + \Re\int_{\Omega_{ab}\times \sphere^2 } F
        (\mu+\nu)^{-2}(\pd[\mu]-\pd[\nu])\overline{w} \; d\mu d\nu
        d\theta \\
        \quad + E(w_0,w_1) + R(w,\mu_0).
      \end{aligned}      \label{energy1}
    \end{align}

    If we write
    \begin{equation*}
      w(\mu,\nu,\theta)=w(\nu,\nu,\theta) - \int_{\mu}^{\nu} \pd[s] w(s,\nu,\theta) \; ds,
    \end{equation*}
    and the Cauchy--Schwarz inequality, we obtain
    \begin{gather*}
      |w(\mu,\nu,\theta)|^{2} \leq 2 |w(\nu,\nu,\theta)|^2+ 2(\nu-\mu)\int_{\mu}^\nu |\p_s w(s,\nu,\theta)|^2 \; ds.
    \end{gather*}
    Integrating this inequality in $\Omega_{ab}$ we have
    \begin{equation}
      \norm[L^{2}(\Omega_{ab}\times \sphere^{2})]{w}^{2} \leq 2(b-a) \left(
        \norm[L^{2}(\Sigma_{3}\times \sphere^{2})]{w}^{2}+
        \norm[L^{2}(\Omega_{ab}\times \sphere^{2})]{\pd[\mu]w}^{2} \right). 
      \label{CS1}
    \end{equation}
    Again the Cauchy--Schwarz inequality gives
    \begin{align}
      \begin{aligned}
        \left|\int_{\Omega_{ab}\times \sphere^2 } G w
          (\mu+\nu)^{-2}(\pd[\mu]-\pd[\nu])\overline{w} \; d\mu d\nu
          d\theta\right|\leq \frac{\mu_0^{-2}}{8} \norm[L^{\infty}]{G}
        \left( \norm[L^{2}(\Omega_{ab}\times\sphere^{2})]{w}^{2} +
          \norm[L^{2}(\Omega_{ab}\times\sphere^{2})]{(\pd[\mu]-\pd[\nu])w}^{2} \right), \\
        \left|\int_{\Omega_{ab}\times \sphere^2 } F
          (\mu+\nu)^{-2}(\pd[\mu]-\pd[\nu])\overline{w} \; d\mu d\nu
          d\theta\right|\leq \frac{\mu_0^{-2}}{8}\left(
          \norm[L^{2}(\Omega_{ab}\times\sphere^{2})]{F}^{2}+
          \norm[L^{2}(\Omega_{ab}\times\sphere^{2})]{(\pd[\mu]-\pd[\nu])w}^{2}\right).
      \end{aligned} \label{CS2}
    \end{align}
    We then deduce from \eqref{energy1}, \eqref{CS1} and \eqref{CS2}
    that there exists $C$ which depends only on $\mu_{0}$, $T$, and $||G||_{L^\infty}$ such that
    \begin{align*}
      \int_{\Sigma_2\times \sphere^2} \left[|\pd[\mu] w|^2+
        |\nabla_{\ms^2} w|^2\right] \; d\mu d\theta \leq &C (
      \norm[L^{2}(\Omega_{ab}\times \sphere^{2})]{\pd[\mu]w}^{2} +
      \norm[L^{2}(\Omega_{ab}\times\sphere^{2})]{\pd[\nu]w}^{2} +
      \norm[L^{2}(\Sigma_{3}\times \sphere^{2})]{w}^{2} \\
      &\quad\quad+E(w_0,w_1)
      + \norm[L^{2}(\Omega_{ab}\times\sphere^{2})]{F}^{2} + R(w,\mu_0)),  \\
      \int_{\Sigma_1\times \sphere^2} \left[|\pd[\nu]
        w|^2+|\nabla_{\ms^2} w|^2 \right] \; d\nu d\theta \leq & C (
      \norm[L^{2}(\Omega_{ab}\times \sphere^{2})]{\pd[\mu]w}^{2} +
      \norm[L^{2}(\Omega_{ab}\times\sphere^{2})]{\pd[\nu]w}^{2} +
      \norm[L^{2}(\Sigma_{3}\times \sphere^{2})]{w}^{2}\\
      &\quad\quad+E(w_0,w_1)
      + \norm[L^{2}(\Omega_{ab}\times\sphere^{2})]{F}^{2} + R(w,\mu_0)).
    \end{align*}
    Integrating the first estimate in $a$ for $a_0\leq a \leq b\leq b_{0},$ and the second one in $b$ for
    $\mu_0 \leq b \leq b_{0}$ we deduce that

    \begin{align}
      \begin{aligned}
        \norm[L^{2}(\Omega_{a_{0}b_{0}}\times\sphere^{2})]{\grad _{\sphere^{2}}w}^{2}
        +
        \norm[L^{2}(\Omega_{a_{0}b_{0}}\times\sphere^{2})]{\pd[\mu]w}^{2}
        &+
        \norm[L^{2}(\Omega_{a_{0}b_{0}}\times\sphere^{2})]{\pd[\nu]w}^{2}
        \leq \\        
        C(b_{0}-a_{0}) \bigg(
        &\norm[L^{2}(\Omega_{a_{0}b_{0}}\times \sphere^{2})]{\pd[\mu]w}^{2} +
        \norm[L^{2}(\Omega_{a_{0}b_{0}}\times\sphere^{2})]{\pd[\nu]w}^{2} +
        \norm[L^{2}(\Sigma_{3}\times \sphere^{2})]{w}^{2} \\
        &+E(w_0,w_1)
        + \norm[L^{2}(\Omega_{a_{0}b_{0}}\times\sphere^{2})]{F}^{2} + R(w,\mu_0)\bigg)
      \end{aligned}
    \end{align}
    If we take $b_{0}$ such that $C(b_{0}-a_0)<\ha,$ we obtain the
    desired inequality in the region $\Omega_{a_0b_{0}}$.  Note that
    this estimate depends only on the length $b_{0}-a_{0}$ and so
    applies to all $\Omega_{a_{0}b_{0}}$ (with the same constant) as
    long as $b_{0}-a_{0} < \frac{1}{2C}$.
    
    To finish the proof, we claim that a similar estimate holds for
    all $b_{0}-a_{0}< T$.  Let $M$ be given by
    \begin{equation}
      \label{eq:def-of-m}
      M = \sup \left\{ m < T: \text{ estimate~\eqref{ineq}
          holds uniformly for all } b_{0}
        - a_{0} \leq m \right\}.
    \end{equation}
    We claim that $M=T$.  Our estimate above shows that $M > 0$.  We
    now show that if $M<T$, then there is an $\epsilon > 0$ so that
    estimate~\eqref{ineq} holds uniformly for $b_{0} - a_{0} \leq
    M+\epsilon$, contradicting~\eqref{eq:def-of-m}.

    Let $C_{0}$ be such that if $0 < a < b < T$ and $b-a\leq M$, then
    equation~\eqref{ineq} holds with constant $C_{0}$.  Suppose $a_{0} <
    a_{1} < b_{0} < b_{1}$ with $b_{0} - a_{0} = b_{1} -a_{1}=M$.  Let
    $R$ be the rectangular region given by
    \begin{equation*}
      \Omega_{a_{0}b_{1}} \setminus \Omega_{a_{0}b_{0}} \cup \Omega_{a_{1}b_{1}}.
    \end{equation*}
    By repeating the above argument, we find that 
    \begin{align*}
      \norm[L^{2}(R\times\sphere^{2})]{\grad _{\sphere^{2}}w}^{2}
      +
      \norm[L^{2}(R\times\sphere^{2})]{\pd[\mu]w}^{2}
      &+
      \norm[L^{2}(R\times\sphere^{2})]{\pd[\nu]w}^{2}
      \leq \\        
      C(b_{0}-a_{0}) \bigg(
      &\norm[L^{2}(\Omega_{a_{0}b_{1}}\times \sphere^{2})]{\pd[\mu]w}^{2} +
      \norm[L^{2}(\Omega_{a_{0}b_{1}}\times\sphere^{2})]{\pd[\nu]w}^{2} +
      \norm[L^{2}(\Sigma_{3}\times \sphere^{2})]{w}^{2} \\
      &+E(w_0,w_1)
      + \norm[L^{2}(\Omega_{a_{0}b_{1}}\times\sphere^{2})]{F}^{2} +
      R(w,\mu_0)\bigg) .
    \end{align*}
    We now use that
    \begin{equation*}
      \norm[L^{2}(\Omega_{a_{0}b_{1}})]{\cdot}^{2} \leq
      \norm[L^{2}(\Omega_{a_{0}b_{0}})]{\cdot}^{2} +
      \norm[L^{2}(\Omega_{a_{1}b_{1}})]{\cdot}^{2} +
      \norm[L^{2}(R)]{\cdot}^{2}
    \end{equation*}
    to obtain an estimate on $R\times \sphere^{2}$.  Adding this to
    the estimates on $\Omega_{a_{0}b_{0}}$ and $\Omega_{a_{1}b_{1}}$
    yields an estimate for $\Omega_{a_{0}b_{1}}$ that is uniform in
    $a_{0}$ and $b_{1}$, provided that $b_{1} - b_{0}$ and
    $a_{1}-a_{0}$ are small.  This implies that~\eqref{ineq} holds
    uniformly in $a$ and $b$ provided $b-a\leq M+\epsilon$, a
    contradiction.

    The last statement of the lemma follows from a Poincar{\'e}-type
    inequality.  Indeed, we write $w(\mu, \mu , \theta) = \int
    _{0}^{\mu}(\pd[\mu]w + \pd[\nu]w)(s,s,\theta) \ds$, apply
    Cauchy--Schwarz, and integrate to bound
    $\norm[L^{2}(\tilde{\Sigma}_{3}\times \sphere^{2})]{w}^{2}$.
  \end{proof}

  Now we return to the semilinear wave equation and we multiply
  \eqref{CP2} by $\frac{1}{2}(\mu+\nu)^{-2}(\pd[\mu]-\pd[\nu])
  \overline{v}$ and add the result to its complex conjugate to obtain
  the following identity:
  \begin{align}
    \begin{aligned}
      \frac{1}{2} \pd[\nu]\left[ |\pd[\mu] v|^2 +(\mu+\nu)^{-2}
        |\nabla_{\sphere^2} v|^2 +
        4 (\mu\nu)^2(\mu+\nu)^{-4}\tilde{P}(x,v)\right]- \\
      \frac{1}{2} \pd[\mu]\left[ |\pd[\nu] v|^2 +(\mu+\nu)^{-2}
        |\nabla_{\sphere^2} v|^2 +
        4 (\mu\nu)^2(\mu+\nu)^{-4}\tilde{P}(x,v)\right]+ \\
      4 (\mu\nu)(\nu-\mu)(\mu+\nu)^{-4} \left( \tilde{P} +
        \frac{\mu\nu}{\mu + \nu}(\pd[x]\tilde{P})\right) +
      \Re\diver_{\sphere^2}\left[ \left(
          (\mu+\nu)^{-2}(\pd[\mu]-\pd[\nu]) \overline{v}\right)
        \nabla_{\sphere^2} v\right] =0
    \end{aligned}\label{multid2}
  \end{align}
  We then integrate \eqref{multid2} in $\Omega_{ab}\times \sphere^2$:
  \begin{align*}
    \frac{1}{2} \int_{\Sigma_1\times \sphere^2} (\mu+\nu)^{-2}\left(
      (\mu+\nu)^{2}|\pd[\nu] v|^2 + |\nabla_{\sphere^2} v|^2 +
      (\frac{2\mu\nu}{\mu+\nu})^{2} \tilde{P}(x,v)\right)\; d\nu d\theta+\\
    \frac{1}{2} \int_{\Sigma_2\times \sphere^2} (\mu+\nu)^{2}\left(
      (\mu+\nu)^{2}|\pd[\mu] v|^2 + |\nabla_{\sphere^2} v|^2 +
      (\frac{2\mu\nu}{\mu+\nu})^{2}\tilde{P}(x,v)\right)\; d\mu d\theta \\
    + \int_{\Omega_{ab}\times \sphere^2}\left[ 4
      (\mu\nu)(\nu-\mu)(\mu+\nu)^{-4} \left( \tilde{P}(x,v) +
        \frac{\mu\nu}{\mu+ \nu}(\pd[x]\tilde{P})(x,v)\right)\right] \;
    d\mu d\nu
    d\theta=\\
    \int_{\Sigma_3 \times \sphere^2} (2\mu)^{-2}\left( 4
      \mu^2|\pd[\mu] \phi_0|^2 +4 \mu^{-4}| \phi_1|^2+
      |\nabla_{\sphere^2} \phi_0|^2+ 
      \mu^{-2}\tilde{P}(x,\phi_0)\right) \; d\mu d\theta
  \end{align*}
  In particular this implies that $\pd[\mu] v, \pd[\nu] v,
  \nabla_{\sphere^2} v\in L^2(\Omega_{ab} \times \sphere^2)$ and so
  $v\in H^{1}$.
  
  Let $\Omega _{T} = [0,T]\times [0,T] \times \sphere^{2}$.  We now
  show by induction that $v\in H^{k}(\Omega_{T})$ for all $k$.  We
  know from the above that $v\in H^{1}(\Omega_{T})$.  Suppose now for
  induction that $v\in H^{k}(\Omega_{T})$.  In particular, $\pd
  ^{k}v\in L^{2}(\Omega_{T})$.  The dimension of $\Omega_{T}$ is $4$
  and so we may then use Sobolev embedding to see that $\pd^{k-1}v \in
  L^{4}(\Omega_{T})$, $\pd^{k-2}v\in L^{p}(\Omega_{T})$ for any $p\in
  (1,\infty)$, and $\pd^{\alpha}v\in L^{\infty}(\Omega_{T})$ for all
  $0 \leq |\alpha| \leq k-3$.  (We know already that $v\in
  L^{\infty}(\Omega_{T})$.)

  Let $F = \left( \frac{2\mu\nu}{\mu + \nu} \right)^{2}$.  We
  differentiate equation~\eqref{CP2} $k$ times with respect to the
  angular variables to find that
  \begin{equation*}
    (\mu + \nu)^{2}\pd[\mu]\pd[\nu]\grad_{\sphere^{2}}^{k}v +
    \lap_{\sphere^{2}}\grad_{\sphere^{2}}^{k} v = F
    \grad_{\sphere^{2}}^{k} \tilde{f}(x,v).
  \end{equation*}
  By using Sobolev embedding we ensure that the right hand side is in
  $L^{2}(\Omega_{T})$, so Lemma~\ref{enerest} implies that
  $\grad_{\sphere^{2}}^{k}v \in H^{1}$.  In particular,
  $\lap_{\sphere^{2}}\grad_{\sphere^{2}}^{k-1}v \in L^{2}(\Omega_{T})$
  and so we may differentiate equation~\eqref{CP2} $k-1$ times in
  $\theta$ to find that $\pd[\mu]\pd[\nu]\grad_{\sphere^{2}}^{k-1}v
  \in L^{2}$.  We now differentiate equation~\eqref{CP2} $k-1$ times
  in $\theta$ and once in $\mu$ (or $\nu$) to see that if $w_{k-1,1} =
  \grad^{k-1}_{\sphere^{2}}\pd[\mu]v$, then 
  \begin{equation*}
    (\mu + \nu)^{2}\pd[\mu]\pd[\nu]w_{k-1,1} + 2(\mu +
    \nu)\pd[\nu]w_{k-1,1} + \lap_{\sphere^{2}}w_{k-1,1} =
    \pd[\mu]\grad^{k-1}_{\sphere^{2}}\left( F\tilde{f}(x,v)\right).
  \end{equation*}
  We know already that $\pd[\nu]w_{k-1,1}\in L^{2}$, and the right
  hand side is in $L^{2}$ by Sobolev embedding and the induction
  hypothesis, so $w_{k-1,1}\in H^{1}$.  In particular,
  $\lap_{\sphere^{2}}\grad^{k-2}_{\sphere^{2}}\pd[\mu] v \in L^{2}$
  and so $\pd[\nu]\pd[\mu]^{2}\grad^{k-2}_{\sphere^{2}}v\in L^{2}$.
  We may continue in this fashion to see that $v\in H^{k+1}$.  

  We now have that $v\in H^{k}(\Omega_{T})$ for all $k$.  $\Omega_{T}$
  is bounded and so in fact $v$ is smooth on the closure of
  $\Omega_{T}$.

\end{proof}

\section{ Support theorems for semilinear radiation fields} \label{sec:suppthm}

We will prove  a support theorem for radiation fields with radial initial data.  Finite speed of propagation says that if $u$ is a Shatah-Struwe solution of \eqref{CP},   and if  $\phi$ and $ \psi$ are  supported in $r \leq R,$  then $\mathcal{L}_+(0, \psi)$ is supported in $s \geq -R$.  We are interested in the converse of this statement.   We begin by considering the case of the linear equation.  It is well known, see \cite{helgason:1999}, that there exist $\psi\in L^2(\mr^3)$ not supported in $|z|\leq R,$ but such that  $\mcr(0,\psi,0)=0$ if $|s|\geq R.$  However  if $\psi$ is radial we have the following result, which can be found in \cite{helgason:1999}:

\begin{prop}
  \label{radial-data1} 
  Suppose $\psi \in L^2(\mr^3)$ is a radial function and
  $\mcr_+(0,\psi,0)(s,\theta)=0$ if $s\leq -R.$ Then, $\psi(z)=0$ if
  $|z|\geq R.$  
  
  If $\phi \in \hone (\reals^{3})$ is a radial function,
  $\mcr_{+}(\phi , 0 , 0 )(s,\theta) = 0$ for $|s| \geq R$, and $\int
  _{\reals}\mcr_{+}(\phi, 0, 0)(s,\theta) \ds = 0$, then $\phi (z) =
  0$ for $|z| \geq R$.
\end{prop}

\begin{proof} 
  From \eqref{formula1}
  \begin{equation*}
    \mathcal{R}_+(0,\psi,0)(s,-\theta)=-\frac{1}{4\pi} \pd[s] \int_{\langle z,-\theta\rangle=s} \psi(z) \; d\sigma(z).
  \end{equation*}

  Let us assume for a moment that $\psi$  is $C^\infty$ and that it has compact support.  We use polar coordinates to represent a point  $z$ on the plane $\langle z, \theta\rangle=-s,$ and write it as $z= -s\theta+\rho \gamma,$ where $\gamma \in \sphere^2$ is orthogonal to $\theta.$  Hence $\psi(|z|)=\psi(\sqrt{s^2+\rho^2})$ and we have
  \begin{equation*}
    \pd[s]\int_{\langle z,-\theta\rangle=s} \psi(|z|) \; d\sigma(z)=2\pi \pd[s] \int_0^\infty \rho \psi(\sqrt{s^2+\rho^2}) \; d\rho=2\pi s \int_0^\infty \pd[\rho] \psi(\sqrt{s^2+\rho^2}) \; d\rho=
    -2\pi s\psi(|s|).
  \end{equation*}
  Therefore, by continuity, we have that for $\psi \in L^2(\reals^3),$ radial,
  $\mathcal{R}_+(0,\psi,0)(s)=\ha s \psi(|s|).$  So if $\psi\in L^2(\reals^3)$ is radial and
  $\mathcal{R}_+(0,\psi,0)(s)=0$ for  $s\leq -R,$ then $\psi(z)=0$ for $|z|\geq R$.

  To prove the other statement, we know from \eqref{formula1} that
  \begin{equation*}
    \mathcal{R}_{+}(\phi, 0 , 0)(s) = -\frac{1}{4\pi} \pd[s]^{2}
    \int_{\langle z, -\theta\rangle = s} \phi (z) d\sigma (z).
  \end{equation*}
  The same argument as above shows that if $\phi$ is smooth and
  compactly supported, then
  \begin{equation*}
    \pd[s]^{2}\int_{\langle z,-\theta\rangle =s} \phi (|z|)d\sigma (z) =
    -2\pi \pd[s] \left( s\phi (|s|)\right).
  \end{equation*}
  By continuity, we thus have that if $\phi \in \hone (\reals^{3})$ is
  radial, then $\mathcal{R}_{+}(\phi, 0 , 0)(s) = \frac{1}{2}
  \pd[s]\left( s\phi (|s|)\right)$.  In particular, if
  $\mathcal{R}_{+}(\phi ,0,0)(s) = 0$ for $|s| \leq R$, then there is
  some constant $C$ so that $\phi (r) = C/r$ for $r\geq R$.  This
  constant vanishes precisely when $\int _{\reals}\mcr _{+}(\phi,
  0,0)(s)\ds$ vanishes and so $\phi (z) = 0$ for $|z|\geq R$.
\end{proof}

From this we obtain 
\begin{cor}
  \label{suppthm0} 
  Let $F\in L^2(\mr)$ be supported in $\{|s|\leq R\}$ and suppose that
  $\int _{\reals}F(s)\ds = 0$. Let $(\phi,\psi)$ be such that
  $\mcr_+(\phi,\psi, 0)=F.$ Then $\phi$ and $\psi$ are radial functions
  with finite energy and $\phi(z)=\psi(z)=0$ if $|z|\geq R.$
\end{cor}
\begin{proof} 
  We have already shown that $\phi$ and $\psi$ must be radial.

  $F(s)=\mcr_+(\phi,\psi,0)$ and let $F^*(s)=F(-s),$ then
  $F=F_{e}+F_{o},$ $F_{e}=\ha(F+F^*)$ and $F_{o}=\ha(F-F^*).$ Then
  $F_e= \mcr_+(\phi,0,0)$ and $F_o=\mcr_+(0,\psi,0)$ are both
  supported in $[-R,R]$ and $\int F_{e} = 0$, so by Proposition \ref{radial-data1},
  $\phi$ and $\psi$ are supported in $\{|z|\leq R\}.$ 
\end{proof}

We will prove that a weaker result holds for the semilinear equation as well.
\begin{thm}
  \label{L2suppthm} 
  Let $F\in L^2(\mr)$ be such that $F(s)=0$ for $|s|\geq R$ and $\int
  F(s)\ds = 0$. Let $(\phi,\psi)$ be such that $\mcl_+(\phi,\psi)=F.$
  Then $\phi$ and $\psi$ are radial functions with finite energy and
  compact support.
\end{thm}
\begin{proof}   
  We argue as in in Corollary \ref{suppthm0}, to show that $\phi$ and
  $\psi$ are radial.  Indeed, since $F$ does not depend on $\theta,$
  we get that $\mcl_+(\phi,\psi)= \mcl_+(U^*\phi, U^*\psi).$ By the
  injectivity of the map $\mcl_+$ we find that $\phi=U^*\phi$ and
  $\psi=U^*\psi.$

  Now we follow the construction of the data $(\phi,\psi)$ in the proof of Theorem \ref{L2bound-radf}.   We know that there exists $(\phi_0,\psi_0)$ with finite energy such that $\mcr_+(\phi_0,\psi_0,0)=F.$   We also know from Corollary \ref{suppthm0} that $\phi_0, \psi_0$ are radial and supported in $\{|z|\leq R\}.$

  Let $v$ satisfy \eqref{CP0} with initial data $(\phi_0,\psi_0)$ and $f=0.$  Let  $\del>0$ and let $T_0$ be such that \eqref{choiceofdelta} holds. By finite speed of propagation, $v(T_0,z)$ is supported in the ball $\{|z|\leq T_0+R\}.$ Now we modify the definition of the space $B_\delta$ to control the support of the solution, and set 
  \begin{equation*}
    \wtb_\del= B_\del \cap \{w: w(t,z)  \text{ is supported in }  t-T_0 \geq |z|-R-T_0, \;\ t>0\}.
  \end{equation*}
  It follows that $\wtb$ is also a closed Banach subspace of
  $L^5([T_0,\infty); L^{10}(\mr^3)),$ and therefore the same iteration
  scheme of the proof of Theorem \ref{L2bound-radf} goes through, we may
  pick $\phi_w,\psi_w$ to be supported in $\{|z| \leq T_{0}+R\}$ and we
  find a solution $u$ to \eqref{auxCP1} with data supported on $\{|z|
  \leq T_{0}+R\}.$ Again, as in the proof of Theorem \ref{L2bound-radf} we
  use the result of Shatah-Struwe to solve the semilinear Cauchy problem
  backwards.  By finite speed of propagation we find that the initial
  data $(\phi,\psi)= ( u(0,z), \p_t u(0,z))$ is supported in $\{|z|\leq
  R+2T_{0}\}$ and $\mcl(\phi,\psi)=F.$
\end{proof}

\begin{thm}
  \label{thm:smoothsuppthm}
  Let $F\in C^{\infty}(\reals)$ be such that $F(s) = 0$ for $|s| \geq
  R$ and $\int _{\reals} F(s) \ds = 0$.  Let $(\phi, \psi)$ be such
  that $\mcl_{+} (\phi, \psi) = F$.  Then $\phi$ and $\psi$ are smooth
  and compactly supported.
\end{thm}

\begin{proof}
  We know already that $F = \mcl_{+}(\phi, \psi)$, where $\phi$ and
  $\psi$ are radial, have finite energy, and are compactly supported.
  It remains to show that $\phi$ and $\psi$ are smooth.

  We know that there are smooth compactly supported
  radial functions $\phi _{0}$ and $\psi _{0}$, supported in $\{
  |z|\leq R\}$, such that $F = \mathcal{R}_{+}(\phi_{0}, \psi_{0},
  0)$.  By applying the iteration scheme from Theorem~\ref{L2suppthm}
  with the $L^{5}W^{k,10}$ norm (which is finite by
  Proposition~\ref{thm:higher-reg-NL}), and then using persistence of
  regularity and uniqueness,
  we find that $(\phi , \psi) \in \tilde{H}^{k+1} \times H^{k}$.  This
  is true for all $k$, so $\phi$ and $\psi$ are smooth.
\end{proof}

\begin{remark}
  \label{rem:scattering-op-is-also-cpct-supp}
  A consequence of Theorem~\ref{thm:smoothsuppthm} (and the finite
  speed of propagation) is that for such $F$, $\mathcal{A}F$ is
  smooth, radial, and vanishes for $s$ sufficiently negative.
\end{remark}

If we know more about the regularity of the initial data, we may drop
some of the hypothesis on the function $F$ from Theorem
\ref{L2suppthm} and control the support of the initial data.

\begin{thm}
  \label{supportthm1}   
  Let $\psi(z)=\psi(|z|)\in C_0^\infty(\reals^3)$ be a radial
  function. If $\mcl_+(0,\psi)(s)=0$ if $s \leq -R$, then $\psi(z)=0$
  if $|z| \geq R.$ Moreover, if $(\phi, \psi) \in
  C_{0}^{\infty}(\reals^{3})$ are radial functions, and both
  $\mathcal{L}_{\pm}(\phi, \psi) (s) = 0$ for $s \leq -R$, then both
  $\phi(z)=0$ and $\psi(z) = 0$ for $|z|\geq R$.
\end{thm}

\begin{proof} 
  The proof is be an application of unique continuation results and
  is based on the methods of \cite{Sa-Barreto:2003}.  The counter-examples of
  Alinhac \cite{alinhac:1983} and  Alinhac and Baouendi \cite{alinhac:1995},
  indicate that this result is unlikely to be true if $\psi$ is not
  assumed to be radial, but this is an open problem.
  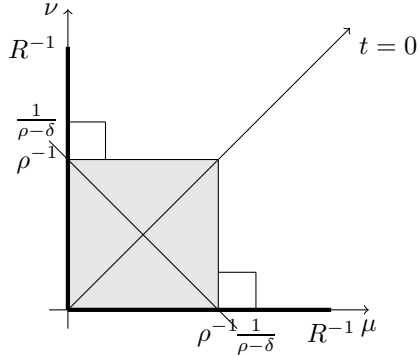
\begin{figure}[htp]
    \centering
    \begin{tikzpicture}
      \draw [->] (-0.25,0) -- (4,0) node[anchor=north](muaxis){$\mu$};
      \draw [->] (0, -0.25) -- (0,4) node[anchor=east](nuaxis){$\nu$};
      
      \fill [gray!20!white] (0,0) rectangle (2,2);

      \draw [ultra thick] (0,0) -- (3.5,0) node[anchor=north](muR){$R^{-1}$};
      \draw [ultra thick] (0,0) -- (0,3.5) node[anchor=east](nuR){$R^{-1}$};

      \draw [-] (2,0) node[anchor=north](murho){$\rho^{-1}$} -- (2,2);
      \draw [-] (0,2) node[anchor=east](nurho){$\rho^{-1}$} -- (2,2);

      \draw [-] (2.5,0) node[anchor=north]{$\frac{1}{\rho - \delta}$}-- (2.5,.5);
      \draw [-] (2,.5) -- (2.5,.5);
      \draw [-] (0,2.5) node[anchor=east]{$\frac{1}{\rho-\delta}$}-- (.5,2.5);
      \draw [-] (.5,2) -- (.5,2.5);

      \draw [->] (0,0) -- (3.75,3.75) node[anchor=north
      west](diag){$t=0$};
      
      \draw (-0.25, 2.25) -- (2.25, -0.25);

    \end{tikzpicture}
    \caption{The regions in Theorem~\ref{supportthm1}.  The function $v$
      vanishes in the grey box and along the dark lines.  The line $\mu
      + \nu = \rho ^{-1}$ is the surface on which the unique continuation
      argument is applied.}
    \label{fig:regions-supp-mu-nu}
  \end{figure}

  Let us assume that  $\phi=0$ and $\psi(z)=0$ if $|z|\geq\rho>R.$ By finite speed
  of propagation, the solution to \eqref{CP} satisfies
  \eqref{finitespeed1} and \eqref{finitespeed} with $R$ replaced by
  $\rho.$ Since $u$ is radial, then from \eqref{CP2}, we find that
  $v(\mu,\nu)$ is $C^\infty$ and satisfies
  \begin{gather}
    \begin{gathered}
      (\mu+\nu)^2\pd[\mu] \pd[\nu]  v- \left(
        2\frac{\mu\nu}{\mu+\nu}\right)^{2} \tilde{f}(x,v)=0, \text{ in }
      (0,T)\times (0,T)  \\ 
      v(\mu,\mu,\theta)= 0 \;\ (\pd[\mu] v)(\mu,\mu,\theta)=
      \frac{1}{2} \mu^{-3} \psi(\mu\theta).
    \end{gathered}    \label{CP3}
  \end{gather}
  
  By assumption $\mcl_+(0,\psi)(s)=0$ for $s\leq -R.$ Then in
  coordinates $(\mu,\nu)$ this implies that $v,$ the solution to
  \eqref{CP3} with data $(0,\mu^{-3}\psi)$ satisfies
  \begin{gather}
    v(0,\nu)=0, \;\ \nu\leq \frac{1}{R}. \label{region1mu}
  \end{gather}
  Since the non-linearity is odd, we have that the solution $v$ to \eqref{CP} is also odd.  In coordinates $(\mu,\nu)$ this implies that $v(\mu,\nu)=-v(\nu,\mu),$ and therefore we conclude that 
    \begin{gather}
    v(\mu,0)=0, \;\ \mu\leq \frac{1}{R}. \label{region1mu1}
  \end{gather}
  
  Now substituting this into \eqref{CP3}, and using that $v(\mu,\nu)=0$
  near $\mu=\nu=0,$ we find that in fact for every $k\in \mn,$
  \begin{gather*} 
    \p_\mu^k v(0,\nu)=0, \;\ \nu\leq \frac{1}{R} \text{ and } \p_\nu^k v(\mu,0)=0, \;\ \mu\leq \frac{1}{R}.
  \end{gather*}
  Therefore we can extend $v$ as $v=0$ if  
  \begin{gather}
    \{\mu<0, \nu\leq \frac{1}{R}\} \cup \{\nu<0, \mu\leq \frac{1}{R}\}. \label{region2mu}
  \end{gather}
  $v=0$ in the union of the regions \eqref{region1mu} and
  \eqref{region2mu}.  We want to use a unique continuation result to
  guarantee that $v=0$ in a neighborhood of $(0,\rho^{-1})$ and
  $(\rho^{-1},0).$ We need to transform \eqref{CP3} into a linear
  equation.  As above, we differentiate \eqref{CP} with respect to $t$
  and we get that $w=\p_t v$ satisfies in coordinates $(\mu,\nu)$
  \begin{align*}
    (\mu+\nu)^2 \p_\mu\p_\nu w&-W_{1}w - W_{2}\overline{w} =0, \text{
      with } \\
    W_{1} &= \left( \frac{2\mu\nu}{\mu +
        \nu}\right)^{2}x^{-4}\left[ f_{0}(|xv|^{2}) +
      f_{0}'(|xv|^{2})|xv|^{2}\right] \text{ and } \\
    W_{2} &=
    \left( \frac{2\mu\nu}{\mu + \nu}\right)^{2u}x^{-4}\left[ f_{0}'(|xv|^{2})(xv)^{2}\right].
  \end{align*}
  Note that in the above, our assumptions on $f$ (and hence $f_{0}$)
  imply that $x^{-4}f_{0}(|xv|^{2})$ and
  $x^{-4}f_{0}'(|xv|^{2})|xv|^{2}$ are smooth functions of $x$ and
  $v$.

  Using that $w=0$ in the union of the regions \eqref{region1mu} and
  \eqref{region2mu}, then in particular $w=0$ if $\mu+\nu\leq
  \rho^{-1}$, and this implies that $w=0$ near $(0,\rho^{-1})$ and
  $(\rho^{-1},0).$ Let us say, $w(\mu,\nu)=0$ if $\mu< \delta$ and
  $\nu\leq \frac{1}{\rho-\del}$ and similarly if $\nu<\delta$ and
  $\mu\leq \frac{1}{\rho-\del}.$ In terms of the variables $s_+$ and
  $s_-$ this implies that
  \begin{gather*}
    w=0 \text{ if }  s_+\leq \del-\rho,   \;\ s_- \geq \frac{1}{\del}, \\
    w=0 \text{ if }  s_-\geq  \rho-\del,   \;\ s_+ \leq -\frac{1}{\del}.
  \end{gather*}
  In particular this implies that
  \begin{gather*}
    w=0 \text{ if } r  \leq \frac{1}{\del}, \;\ \del-\rho-r \leq t \leq r + \del-\rho 
  \end{gather*}
  Using the hyperbolicity of $\p_r^2-\p_t^2$ with respect to $r$ or
  $t$ gives that the initial data vanishes if $r\geq \rho-\del.$
  Proceeding this way, we find that $\psi$ is supported in $r\leq R.$

  The argument with initial data $(\phi, \psi)$ is nearly
  identical---the additional condition on $\mathcal{L}_{-}(\phi,
  \psi)$ acts as a replacement for the assumption that $\phi = 0$ in
  order to guarantee that both $\mathcal{L}_{\pm}(\phi, \psi)(s)$
  vanish for $s \leq -R$.  The rest of the proof proceeds with minimal
  changes.
\end{proof}

Putting the above together proves Theorem~\ref{thm:fullsupportthm}.

\appendix

\section{Persistence of regularity}
\label{sec:pers-regul}

In this section we outline a proof of persistence of regularity and
show that the $L^{5}W^{k,10}$ and $L^{4}W^{k,12}$ norms of solutions
of equation~\eqref{CP} are bounded.  We believe this is well-known but we
include it for completeness.




As before, we define the higher energy spaces $\tilde{H}^{k+1}$ and $H^{k}$
for $k \geq 1$ with the following norms
\begin{align*}
  \norm[\tilde{H}^{k}]{u}^{2} &=\sum _{j=1}^{k}
  \norm[\dot{H}^{k}]{u}^{2}, \\
  \norm[H^{k}]{u}^{2} &= \norm[\tilde{H}^{k}]{u}^{2} + \norm[L^{2}]{u}^{2}.
\end{align*}
The $\tilde{H}^{k}$ norms omit the $L^{2}$ portion of the standard inhomogeneous
Sobolev norms.

By commuting the wave operator with powers of the Laplacian we 
obtain the following energy and Strichartz estimates.
\begin{prop}
  \label{prop:higher-order-energy-strichartz}
  Suppose that $\Box u = f\in L^{1}H^{k}$ and $(u, \pd[t]u)|_{t=0} =
  (\phi, \psi) \in \tilde{H}^{k+1}\times H^{k}$.  Then $u$ satisfies the
  following estimate:
  \begin{equation*}
    \norm[L^{\infty}\tilde{H}^{k+1}]{u} +
    \norm[L^{\infty}H^{k}]{\pd[t]u} + \norm[L^{5}W^{k,10}]{u} +
    \norm[L^{4}W^{k,12}]{u} \leq C \left( \norm[\tilde{H}^{k+1}]{\phi} +
      \norm[H^{k}]{\psi} + \norm[L^{1}H^{k}]{f}\right)
  \end{equation*}
\end{prop}

We now show that if $u$ is a Shatah-Struwe solution of
equation~\eqref{CP} with initial data $(\phi, \psi) \in
\tilde{H}^{k+1}\times H^{k}$, then $u$ has more regularity (and finite
global Strichartz norms).
\begin{thm}
  \label{thm:higher-reg-NL}
  If $u$ solves equation~\eqref{CP} with initial data $(\phi, \psi)
  \in \tilde{H}^{k+1}\times H^{k}$, then $u$ has more regularity,
  i.e., the norms
  \begin{equation*}
    \norm[L^{\infty}\tilde{H}^{k+1}]{u},
    \norm[L^{\infty}H^{k}]{\pd[t]u}, \norm[L^{5}W^{k,10}]{u}, \text{
      and }\norm[L^{4}W^{k,12}]{u} 
  \end{equation*}
  are finite.
\end{thm}
\begin{proof}
  We rely on the estimates from
  Proposition~\ref{prop:higher-order-energy-strichartz} and an
  additional energy estimate.  We proceed via induction.

  Let us first show the claim for $k=1$.  We start by defining
  \begin{equation*}
    E_{2}(t)^{2} = \frac{1}{2}\sum_{\alpha = 1}^{3} \int \left(\left| \pd[\alpha] \grad u\right| ^{2}
      + \left( \pd[\alpha] \pd[t] u \right)^{2} \right)+ \frac{1}{2}\int \left( \left| \grad
        u\right|^{2} + \left( \pd[t]u\right)^{2}\right).
  \end{equation*}
  Differentiating this expression and using equation~\eqref{CP} yields
  that
  \begin{equation*}
    E_{2}(t)E_{2}'(t) \leq C E_{2}\left(
      \norm[L^{12}_{z}]{u(t)}^{4} \norm[L^{6}_{z}]{\pd[\alpha]u(t)} +
      \norm[L^{12}_{z}]{u(t)}^{4}\norm[L^{6}_{z}]{u(t)}\right) .
  \end{equation*}
  Sobolev embedding then shows that
  \begin{equation*}
    \frac{E_{2,}'(t)}{E_{2}(t)} \leq C'\norm[L^{12}_{z}]{u(t)}^{4},
  \end{equation*}
  and so integrating both sides yields that
  \begin{equation*}
    E_{2}(t) \leq E_{2}(0) e^{C'\norm[L^{4}L^{12}]{u}^{4}}.
  \end{equation*}
  We know already by the Strichartz estimates in Theorem~\ref{GV} that
  $u\in L^{4}L^{12}$, so we may conclude that if $(\phi, \psi)\in
  H^{2}\times H^{1}$, then $u \in L^{\infty}H^{2}$.  We now use the
  energy estimate from
  Proposition~\ref{prop:higher-order-energy-strichartz} to bound
  \begin{equation*}
    \norm[L^{4}W^{1,12}]{u} + \norm[L^{5}W^{1,10}]{u} \leq C\left(
      \norm[H^{2}]{\phi} + \norm[H^{1}]{\psi} + \norm[L^{1}H^{1}]{f(u)}\right).
  \end{equation*}
  We note that $\norm[L^{1}H^{1}]{f(u)} \leq C\left(
    \norm[L^{1}L^{2}]{|u|^{4}u} + \sum_{\alpha =
      1}^{3}\norm[L^{1}L^{2}]{|u|^{4}\pd[\alpha]u}\right)$.  The first
  of these two terms is bounded by the result of Bahouri and
  G{\'e}rard, while the second term is bounded by
  $\norm[L^{4}L^{12}]{u}^{4}\norm[L^{\infty}\tilde{H}^{2}]{u}$ via
  H{\"o}lder's inequality and Sobolev embedding and so both are
  finite.  This finishes the proof for $k=1$.
  
  For the inductive step, we suppose that
  $\norm[L^{\infty}\tilde{H}^{k}]{u}$,
  $\norm[L^{\infty}H^{k-1}]{\pd[t]u}$, $\norm[L^{5}W^{k-1,10}]{u}$,
  and $\norm[L^{4}W^{k-1,10}]{u}$ are all finite.
  \begin{equation*}
    E_{k+1}(t)^{2} = \frac{1}{2}\sum_{|\alpha|\leq k-1}\int \left| \pd[\alpha]\grad u\right|^{2} +
    \left( \pd[\alpha] \pd[t]u\right)^{2} + E_{k}^{2}.
  \end{equation*}
  Differentiating both sides shows that
  \begin{equation*}
    E_{k+1}(t) E_{k+1}'(t) \leq C E_{k+1}(t) \norm[H^{k-1}]{|u|^{4}u}.
  \end{equation*}
  A similar argument to the above allows us to bound
  $\norm[H^{k-1}]{\pd[\alpha]f(u)}$ by
  $C\norm[W^{k-1,12}]{u}^{4}E_{k+1}(t)$, and so
  \begin{equation*}
    E_{k+1}(t) \leq E_{k+1}(0)e^{C\norm[L^{4}W^{k-1,12}]{u}^{4}}.
  \end{equation*}
  This demonstrates that $u\in L^{\infty}\tilde{H}^{k+1}$ and
  $\pd[t]u\in L^{\infty}H^{k}$.  We again use the estimate
  from Proposition~\ref{prop:higher-order-energy-strichartz} to bound
  \begin{equation*}
    \norm[L^{4}W^{k,12}]{u} + \norm[L^{5}W^{k,10}]{u} +
    \norm[L^{\infty}\tilde{H}^{k+1}]{u}  +
    \norm[L^{\infty}H^{k}]{\pd[t]u} \leq C\left(\norm[H^{k+1}]{\phi} +
      \norm[H^{k}]{\psi} + \norm[L^{1}H^{k}]{|u|^{4}u} \right),
  \end{equation*}
  and again use Sobolev embedding to bound $\norm[L^{1}H^{k}]{|u|^{4}u}
  \leq \norm[L^{4}W^{k-1,12}]{u}^{4}\norm[L^{\infty}\tilde{H}^{k+1}]{u} < \infty$.
\end{proof}

\begin{remark}
  \label{rem:higher-t-derivs}
  The above argument also shows that we may bound higher derivatives
  in $t$ by using the equation.
\end{remark}

\bibliographystyle{plain}
\nocite{Melrose:1994,Melrose:1995,Sa-Barreto:2005a,Sa-Barreto:2008}
\nocite{tao:2006,tao:2006a}
\bibliography{papers}

\end{document}